\documentclass[a4paper,12pt]{amsart}
\usepackage{amsmath}
\usepackage{amsmath}
\usepackage{amssymb}
\usepackage{amscd}
\usepackage{mathrsfs}
\usepackage[all]{xy}
\usepackage[dvipdfm]{graphicx}
\def\mathcal{\mathscr}
\everymath{\displaystyle}
\newtheorem{thm}{Theorem}[section]
\newtheorem{lem}[thm]{Lemma}
\newtheorem{cor}[thm]{Corollary}
\newtheorem{prop}[thm]{Proposition}

\theoremstyle{definition}

\newtheorem{rem}[thm]{Remark}

\newtheorem{defn}[thm]{Definition}

\newcommand{\mca}[1]{{\mathcal{#1}}}
\def\ad{\text{\rm ad}}

\def\Z{{\mathbb Z}}
\def\C{{\mathbb C}}

\def\R{{\mathbb R}}

\def\all{\textit{\rm all}}
\def\capp{\textit{\rm cap}}

\def\const{\text{\rm const}}
\def\Crit{\text{\rm Crit}}

\def\CZ{\text{\rm CZ}}
\def\diam{\text{\rm diam}}
\def\dist{\text{\rm dist}}

\def\GL{\text{\rm GL}}

\def\HF{\text{\rm HF}}

\def\ind{\text{\rm ind}}
\def\inn{\text{\rm in}}
\def\interior{\text{\rm int}}

\def\Morse{\text{\rm Morse}\,}

\def\rest{\text{\rm rest}}

\def\rk{\text{\rm rk}}

\def\SH{\text{\rm SH}}
\def\Sp{\text{\rm Sp}}

\def\vol{\text{\rm vol}}
\begin{document}
\pagestyle{plain}
\thispagestyle{plain}

\title[Symplectic capacity and short periodic billiard trajectory]
{Symplectic capacity and short periodic billiard trajectory}

\author[Kei Irie]{Kei Irie}
\address{Department of Mathematics, Faculty of Science, Kyoto University,
Kyoto 606-8502, Japan}
\email{iriek@math.kyoto-u.ac.jp}

\subjclass[2010]{Primary:34C25, Secondary:53D40}
\date{\today}

\begin{abstract}
We prove that a bounded domain $\Omega$ in $\R^n$ with smooth boundary has a periodic billiard trajectory with 
at most $n+1$ bounce times and of length less than $C_n r(\Omega)$, where $C_n$ is a positive constant 
which depends only on $n$, and 
$r(\Omega)$ is the supremum of radius of balls in $\Omega$. 
This result improves the result by C.Viterbo, which asserts that $\Omega$ has a periodic billiard trajectory 
of length less than $C'_n \vol(\Omega)^{1/n}$.
To prove this result, we study symplectic capacity of Liouville domains, which is defined via
symplectic homology.
\end{abstract}

\maketitle

\section{Introduction}\label{sec:intro}
Let $\Omega$ be a bounded domain in ${\R}^n$ with smooth boundary.
A \textit{periodic billiard trajectory} on $\Omega$ is a continuous map $\gamma: \R/\tau\Z \to \bar{\Omega}$ such that 
there exists a finite set $\mca{B} \subset \R/\tau\Z$ and satisfies the following conditions:
\begin{itemize}
\item $\gamma$ is smooth on $(\R/\tau\Z)\setminus \mca{B}$ and satisfies $\ddot{\gamma}=0$.
\item For each $t_0 \in \mca{B}$, $\gamma(t_0) \in \partial \Omega$, the left and right derivatives
$\dot{\gamma}(t_0^\pm):=\lim_{t \to t_0^\pm} \dot{\gamma}(t)$ exist, and satisfy the law of reflection
($\nu$ denotes the outward normal vector on $\partial \Omega$):
\begin{align*}
\big\langle \dot{\gamma}(t_0^+), \nu(\gamma(t_0)) \big\rangle
&=-\big\langle \dot{\gamma}(t_0^-), \nu(\gamma(t_0)) \big\rangle \ne 0,\\
\dot{\gamma}(t_0^+)- \big\langle \dot{\gamma}(t_0^+), \nu(\gamma(t_0)) \big\rangle \cdot \nu(\gamma(t_0))
&=\dot{\gamma}(t_0^-) -\big\langle \dot{\gamma}(t_0^-), \nu(\gamma(t_0)) \big\rangle \cdot \nu(\gamma(t_0)).
\end{align*}
\end{itemize}
Elements of $\mca{B}$ are called \textit{bounce times}.

Before stating the main theorem, we introduce some notations. 
For $x \in \R^n$ and $r \ge 0$, $B(x,r):=\bigl\{ y \in \R^n \bigm{|} |x-y| \le r\bigr\}$.
For $\Omega \subset \R^n$, 
\[
r(\Omega):=\sup \bigl\{r \ge 0\bigm{|}\text{there exists $x \in \R^n$ such that $B(x,r) \subset \Omega$}\bigr\}.
\]
The main theorem of this paper is the following:

\begin{thm}\label{mainthm}
Let $\Omega$ be a bounded domain in $\R^n$ with smooth boundary.
Then, there exists a periodic billiard trajectory $\gamma$ on $\Omega$ with at most $n+1$ bounce times and
which satisfies the following length estimate:
\[
|\gamma| \le C_n r(\Omega),
\]
where $C_n$ is a positive constant which depends only on $n$.
\end{thm}

\begin{rem}
The existence of a periodic billiard trajectory with at most $n+1$ bounce times is due to 
\cite{BenciGiannoni}.
They also consider arbitrary metrics on $\R^n$.

In \cite{Viterbo2}, it is proved (theorem 4.1) that there exists a periodic billiard trajectory $\gamma$ on 
$\Omega$ (with the flat metric) which satisfies a length estimate 
$|\gamma| \le C'_n\vol(\Omega)^{1/n}$, where $C'_n$ is a positive constant which depends only on $n$.
Notice that this result follows from theorem \ref{mainthm} and an
obvious inequality $r(\Omega) \le \omega_n^{-1/n} \vol(\Omega)^{1/n}$, where $\omega_n$ denotes the volume of the $n$ -dimensional unit ball.

In \cite{AlbersMazzucchelli}, it is proved (theorem 1.2, the case of a constant potential)
that there exists a periodic billiard trajectory $\gamma$ on 
$\Omega$ (with the flat metric) with at most $n+1$ bounce times and which
satisfies a length estimate $|\gamma| \le C'' \diam \Omega$, where 
$C''$ is a constant which does not depend on $n$, and 
\[
\diam \Omega:= \inf \bigl\{ |v| \bigm{|} (v+\Omega) \cap \Omega = \emptyset \bigr\}.
\]
Notice that our main theorem also implies this result for each fixed $n$, though we can not prove the independence 
of $C''$ on $n$ by this argument.

Finally we remark that one can easily construct $(\Omega_k)_k$, a sequence of bounded domains in $\R^n$,
such that 
$\lim_{k \to \infty} \vol(\Omega_k) = \lim_{k \to \infty} \diam(\Omega_k) = \infty$ and 
$r(\Omega_k) \le 1$ for any $k$.
\end{rem}

To prove theorem \ref{mainthm}, we use \textit{symplectic capacity} defined via 
symplectic homology, which was introduced in 
\cite{Viterbo}.
In the present paper, symplectic capacity is defined for Liouville domains
(compact exact symplectic manifolds with convex boundaries), 
and it is denoted by
$\capp_S$. The definition is given at the beginning of section 3.

Using symplectic capacity $\capp_S$, we introduce the notion of capacity for Riemannian manifolds (without boundaries), 
which is denoted by $\capp_R$.
Roughly speaking, it is defined by $\capp_R(N):=\capp_S(DT^*N)$, where 
$DT^*N:=\bigl\{(q,p) \in T^*N \bigm{|} |p| \le 1 \bigr\}$.
But when $N$ is non-compact, the right hand side does not make sence since $DT^*N$ is not a Liouville domain
(since it is not compact).
Hence we have to approximate $DT^*N$ by compact domains.
See definition \ref{defn:width} for the precise definition.

We prove that 
$\capp_R$ satisfies following properties:
\begin{enumerate}
\item[(A)] Let $\Omega$ be a non-empty open set in ${\R}^n$. Then, $\capp_R(\Omega) \le C_n r(\Omega)$, 
where $\Omega$ is equipped with the flat Riemannian metric on $\R^n$.
\item[(B)] If $\Omega$ is a bounded domain in ${\R}^n$ with smooth boundary,
there exists a periodic billiard trajectory on $\Omega$ with at most $n+1$ bounce
times and of length equals to $\capp_R(\Omega)$.
\end{enumerate}
Our main theorem \ref{mainthm} follows at once from (A) and (B).

We explain the structure of this paper.
In section 2, we recall the notion of symplectic homology. 
We use the version introduced in \cite{Viterbo}.

In section 3, we define $\capp_S$, and prove its properties.
The most important result in this section is theorem \ref{thm:covering}, which asserts that 
when $\pi:Y \to X$ is a covering map between Liouville domains, then 
$\capp_S(Y) \le \capp_S(X)$.
Though its proof is not very difficult, it seems to the author that this result contains a novel idea.

In section 4, we define $\capp_R$, and prove its properties. 
The main result in this section is theorem \ref{thm:mainestimate}, which includes the property (A).
Theorem \ref{thm:covering} is used to prove that $\R^n \setminus \Z^n$ with the flat metric 
has a finite capacity (theorem \ref{thm:Rn-Zn}).
Theorem \ref{thm:mainestimate} is proved by theorem \ref{thm:Rn-Zn} and elementary geometric arguments.

In section 5, we prove the property (B) (theorem \ref{thm:billiard}). 
The arguments in this section heavily rely on the techniques developed in the recent paper \cite{AlbersMazzucchelli}.

In the appendix, we prove theorem \ref{thm:truncated}, which asserts that truncated symplectic homology
of a Liouville domain $(X,\lambda)$ depends only on $d\lambda$. 
It seems to the author that theorem \ref{thm:truncated} is well-known to experts.
But we give a proof of the result since the author is unable to find its proof in the literature.

\section{Symplectic homology}\label{sec:SH}
\subsection{Liouville domains}

First, we recall the notion of \textit{Liouville domains}.
A Liouville domain is a pair $(X,\lambda)$ where $X$ is a compact manifold with boundary and 
$\lambda$ is a $1$-form on $X$, with the following conditions:
\begin{enumerate}
\item[(1)] $(X,d\lambda)$ is a symplectic manifold.
\item[(2)] $Z \in \mca{X}(X)$ defined by $i_Z d\lambda=\lambda$ points strictly outwards on $\partial X$.
\end{enumerate}
(2) implies that $(\partial X, \lambda)$ is a contact manifold.
Let $R$ be the Reeb vector field on $(\partial X, \lambda)$ (recall that $R$ is characterized by 
$i_Rd\lambda=0, \lambda(R)=1$).

In the rest of this paper, $(X,\lambda)$ stands for a Liouville domain, and $n$ stands for $\dim X/2$, 
unless otherwise stated.

$\mca{P}(\partial X, \lambda)$ denotes the set of periodic Reeb orbits on $(\partial X, \lambda)$, and
$\mca{P}_0(\partial X, \lambda)$ denotes the set of elements of $\mca{P}(\partial X,\lambda)$ which is 
contractible in $X$.
For each $x \in \mca{P}(\partial X, \lambda)$, its period is denoted by $\tau(x)$, and
\[
\tau(\partial X, \lambda):= \bigl\{ \tau(x) \bigm{|} x \in \mca{P}(\partial X,\lambda) \bigr\}.
\]

It is well-known that $\tau(\partial X, \lambda)$ is a closed null set in $[0,\infty)$.
Define $\delta(\partial X,\lambda):=\min \tau(\partial X, \lambda)$.
It is clear that $\delta(\partial X, \lambda)>0$.

There exists an unique (up to homotopy) almost complex structure on $TX$, which is compatible with $d\lambda$.
In the rest of this paper, we only treat the case $c_1(TX)=0$.

Let $(X,\lambda)$ be a Liouville domain. 
We define $\Phi \colon \partial X \times (0, 1] \to X$ by 
\[
\Phi(z,1) = z, \qquad \partial_r\Phi(z,r)=r^{-1}Z\bigl(\Phi(z,r)\bigr).
\]
It is easy to check that $\Phi^*\lambda=r\pi^*\lambda$, where $\pi:\partial X \times (0,1] \to \partial X$ 
is the projection.
Define $\hat{X}$ and $\hat{\lambda} \in \Omega^1(\hat{X})$ by
\[
\hat{X}:=X \cup_\Phi \partial X \times (0,\infty), \qquad
\hat{\lambda}:= \begin{cases}
                \lambda &( \text{on $X$}) \\
                r\pi^*\lambda &(\text{on $\partial X \times (0,\infty)$})
                \end{cases}.
\]
We call $(\hat{X},\hat{\lambda})$ the \textit{completion} of $(X,\lambda)$.

By definition, there exists a natural embedding $I: \partial X \times (0,\infty) \to \hat{X}$. 
We often identify $\partial X \times (0,\infty)$ with its image via $I$.
For $r \in (0,\infty)$, $X(r)$ denotes the bounded domain in $\hat{X}$ with boundary $\partial X \times \{r\}$, i.e.
\[
X(r) := \begin{cases}
       X \cup \partial X \times [1,r] &(r \ge 1) \\
       X \setminus \partial X \times (r,1] &(r<1)
       \end{cases}.
\]

\begin{defn}\label{defn:isotopic}
Two Liouville domains $(X,\lambda)$, $(X',\lambda')$ are called \textit{equivalent} if and only if there exists a
diffeomorphism $\varphi: X' \to X$ such that $\lambda'=\varphi^*\lambda$.
$(X,\lambda)$, $(X',\lambda')$ are called \textit{isotopic} if there exists a smooth family of Liouville domains 
$(X,\lambda_t)_{0 \le t \le 1}$, such that $\lambda_0=\lambda$ and $(X,\lambda_1)$ is equivalent to $(X',\lambda')$.
\end{defn}

\subsection{Periodic orbits of Hamiltonian flows}

For $H \in C^\infty(\hat{X})$, we define its \textit{Hamiltonian vector field} $X_H$ by $i_{X_H}d\hat{\lambda}=-dH$.
For $H=(H_t)_{t \in \R/\tau\Z}$, a family of Hamiltonians on $\hat{X}$ parametrized by $\R/\tau\Z$, let us denote by
$\mca{P}_\tau(H)$ the set of $x \colon \R/\tau\Z \to \hat{X}$ which is contractible and satisfies
$\partial_t x = X_{H_t}(x)$.
$\mca{P}_1(H)$ is often abbreviated by $\mca{P}(H)$.

For $x \in \mca{P}_\tau(H)$, we define its \textit{Conley-Zehnder index}.
For later purposes, it is necessary to define the Conley-Zenhder index for degenerate periodic orbits.
Hence we have to define the index for degenerate symplectic paths. 
We use the definiton given in \cite{Long}.

First we introduce some notations.
Take a coordinate $(q_1,p_1,\ldots,q_n,p_n)$ on $\R^{2n}$, and 
define
\begin{align*}
\lambda_n&:=\frac{1}{2}\sum_{1 \le j \le n} p_j \wedge dq_j - q_j \wedge dp_j, \\
\omega_n&:=d\lambda_n,\\
\Sp(2n)&:=\bigl\{V \in \GL(2n,\R) \bigm{|} V^*\omega_n = \omega_n \bigr\}.
\end{align*}
In the present paper, $\GL(2n,\R)$ acts on $\R^{2n}$ from right, i.e. we denote the action of $\GL(2n,\R)$ on
$\R^{2n}$ by 
\[
(x_1,\ldots,x_{2n}) \cdot (V_{ij})_{1 \le i,j \le 2n}
:=\Biggl( \sum_{1 \le j \le 2n} x_j V_{1j}, \ldots, \sum_{1 \le j \le 2n} x_j V_{2n,j}\Biggr).
\]

For $\tau>0$, let us denote
\[
\mca{P}_\tau(2n):= \bigl\{ \gamma \in C^0\bigl([0,\tau],\Sp(2n)\bigr) \bigm{|} \gamma(0)=1_{2n} \bigr\}.
\]
We define the index $i: \mca{P}_\tau(2n) \to \Z$ by several axioms.
To spell out the axioms, we introduce more notations.
\begin{itemize}
\item For $\gamma_0, \gamma_1 \in \mca{P}_\tau(2n)$, $\gamma_0 \sim \gamma_1$ if and only if 
there exists $\delta \colon [0,1] \times [0,\tau] \to \Sp(2n)$ such that
$\delta(i,\cdot)=\gamma_i(\cdot)\,(i=0,1)$, $\delta(s,0)=1_{2n}$, and 
$\rk(\delta(s,1)-1_{2n})$ is constant on $s$. 
\item For $\gamma_i \in \mca{P}_\tau(2n_i)\,(i=0,1)$, define
$\gamma_0 \diamond \gamma_1 \in \mca{P}_\tau(2n_0+2n_1)$ by 
$\gamma_0 \diamond \gamma_1(t):= \begin{pmatrix} \gamma_0(t)&0 \\ 0&\gamma_1(t) \end{pmatrix}$.
\item For $\gamma_0, \gamma_1 \in C^0\bigl([0,\tau],\Sp(2n)\bigr)$ such that $\gamma_0(\tau)=\gamma_1(0)$, define 
$\gamma_1 * \gamma_0 \in C^0\bigl([0,\tau],\Sp(2n)\bigr)$ by 
$\gamma_1 * \gamma_0(t) = \begin{cases} \gamma_0(2t) &(t \le \tau/2)\\ \gamma_1(2t-\tau) &(t \ge \tau/2)\end{cases}$.
\item For $\tau>0$ and $\theta \in \R$, define $\varphi_{\tau,\theta} \in \mca{P}_\tau(2)$ by 
$\varphi_{\tau,\theta}(t):=\begin{pmatrix} \cos(t\theta/\tau) & -\sin(t\theta/\tau)\\ \sin(t\theta/\tau)&\cos(t\theta/\tau) \end{pmatrix}$.
\end{itemize}

The index $i:\mca{P}_\tau(2n) \to \Z$ is defined by the following axioms:

\begin{thm}[\cite{Long}, theorem 6.2.7]\label{thm:CZ}
For $\tau>0$, there exists an unique map $i: \bigcup_{n \ge 1} \mca{P}_\tau(2n) \to \Z$ which satisfies the following five
axioms:
\begin{enumerate}
\item[(1)] For $\gamma_0, \gamma_1 \in \mca{P}_\tau(2n)$, $\gamma_0 \sim \gamma_1 \implies i(\gamma_0)=i(\gamma_1)$.
\item[(2)]For $\gamma_i \in \mca{P}_\tau(2n_i)\,(i=0,1)$, $i(\gamma_0 \diamond \gamma_1)=i(\gamma_0)+i(\gamma_1)$.
\item[(3)] For any $\gamma \in \mca{P}_\tau(2)$ satisfying 
$\gamma(\tau)=\begin{pmatrix}1&a\\0&1\end{pmatrix}\,(a=0,\pm 1)$, 
there exists $\theta_0>0$ such that
$i\bigl([\gamma(\tau)\varphi_{\tau,-\theta}]*\gamma\bigr)=i(\gamma)$ for any $\theta \in (0,\theta_0]$.
\item[(4)] For any $\gamma \in \mca{P}_\tau(2)$ satisfying 
$\gamma(\tau)=\begin{pmatrix}1&a\\0&1\end{pmatrix}\,(a=\pm 1)$, 
there exists $\theta_0>0$ such that
$i\bigl([\gamma(\tau)\varphi_{\tau,\theta}]*\gamma\bigr)=i(\gamma)+1$ for any $\theta \in (0,\theta_0]$.
\item[(5)] Define $\gamma_0 \in \mca{P}_\tau(2)$ by 
$\gamma_0(t):= \begin{pmatrix}1+t/\tau&0\\0&\bigl(1+t/\tau\bigr)^{-1}\end{pmatrix}$.
Then, $i(\gamma_0)=0$.
\end{enumerate}
\end{thm}

In \cite{Long}, several equivalent definitions are given (definition 5.4.2, definition 6.1.10).
In particular, definition 6.1.10 in \cite{Long} implies the following useful lemma:

\begin{lem}\label{lem:degenerate}
Let us define the set of non-degenerate symplectic paths by
\[
\mca{P}^*_\tau(2n):=\bigl\{\gamma \in \mca{P}_\tau(2n) \bigm{|} \rk\bigl(\gamma(\tau)-1_{2n}\bigr)=2n \bigr\}.
\]
Then, for any $\gamma \in \mca{P}_\tau(2n)$,
$i(\gamma)= \sup_{U \in N(\gamma)} \inf\{i(\beta) \bigm{|} \beta \in U \cap \mca{P}_\tau^*(2n)\bigr\}$,
where $N(\gamma)$ denotes the set of all open neighborhoods of $\gamma$ in $\mca{P}_\tau(2n)$.
\end{lem}

The following lemma follows at once from the above lemma.

\begin{lem}\label{lem:i-1}
Assume that a sequence $(\gamma_k)_k$ in $\mca{P}_\tau(2n)$ converges to $\gamma$ in $\mca{P}_\tau(2n)$.
Then, $i(\gamma) \le \liminf_{k \to \infty} i(\gamma_k)$.
\end{lem}

Next we define the Conley-Zehnder index $\mu_{\CZ}(x)$ for $x \in \mca{P}_\tau(H)$.
Let $D^2:=\bigl\{z \in \C \mid |z| \le 1 \bigr\}$, and 
take arbitrary $\bar{x} \colon D^2 \to \hat{X}$ such that 
$\bar{x}(e^{2\pi i\theta}) = x(\tau\theta)$ (such $\bar{x}$ exists since $x$ is contractible).
Since $D^2$ is contractible, $\bar{x}^* T\hat{X}$ is a trivial symplectic vector bundle.
Take the following trivialization of symplectic vector bundle:
\[
F: (\R^{2n},\omega_n) \times D^2 \to \bar{x}^*T\hat{X}; \qquad
(v,z) \mapsto \bigl(F_z(v), z \bigr).
\]

Define $\gamma: \R/\tau\Z \to \Sp(2n)$ by 
\[
\gamma(t):= (F_{e^{2\pi i t/\tau}})^{-1} \circ \Phi_t \circ F_1.
\]
where $(\Phi_t)_t$ is the Poincar\'{e} map generetaed by $(X_{H_t})_t$.
Finally, we define $\mu_{\CZ}(x)$ by 
\[
\mu_{\CZ}(x):=i(\gamma).
\]
Since $c_1(TX)=0$, the above definition is independent of choices of $\bar{x}$.
An element $x \in \mca{P}_\tau(H)$ is called \textit{non-degenerate} if and only if 
$\gamma \in \mca{P}_\tau^*(2n)$.

\subsection{Floer homology on Liouville domains}

For $r_0 \ge 1$, let $\mca{H}(X,\lambda:r_0)$ be the set of 
$H=(H_t)_{t \in \R/\Z}$, a family of Hamiltonians on $\hat{X}$ parametrized by $\R/\Z$, with the following property:
\begin{quote}
There exist $a>0$, $b \in \R$ such that 
$H_t(z,r)= ar + b$ for any $(z,r) \in \partial X \times [r_0, \infty)$ and $t \in \R/\Z$.
(we denote $a, b$ by $a_H, b_H$.)
\end{quote}
We denote $\mca{H}(X,\lambda):=\bigcup_{r_0 \ge 1} \mca{H}(X,\lambda:r_0)$.

$H \in \mca{H}(X,\lambda)$ is called \textit{admissible} if all elements 
of $\mca{P}(H)$ are non-degenerate, and $a_H \notin \tau(\partial X,\lambda)$.
$\mca{H}_{\ad}(X,\lambda)$ denotes the set of all admissible 
$H \in \mca{H}(X,\lambda)$.
Note that when $H$ is admissible, then
$\sharp \mca{P}(H) < \infty$.

For $H \in \mca{H}_{\ad}(X,\lambda)$, we define its Floer homology $\HF_*(H)$.
For each $k \in \Z$, let 
$\mca{P}_k(H)$ denote the set of $x \in \mca{P}(H)$ with $\mu_{\CZ}(x)=k$, and let
$C_k(H)$ denote the free $\Z_2$-module over $\mca{P}_k(H)$.

To define the Floer homology, we need to equip $\hat{X}$ with almost complex structures.
For $r_0 \ge 1$, let $\mca{J}(X,\lambda:r_0)$ be the set of 
$J=(J_t)_{t \in \R/\Z}$, a family of almost complex structures on $\hat{X}$ parametrized by $\R/\Z$, 
such that following properties hold for any $t \in \R/\Z$ ($R$ and $\xi$ denote the Reeb vector field and 
the contact distribution on $(\partial X, \lambda)$): 
\begin{itemize}
\item $J_t$ is compatible with $d\hat{\lambda}$.
\item $J_t\bigl(\partial_r(z,r)\bigr)=r^{-1}R(z,r)$ for $(z,r) \in \partial X \times [r_0,\infty)$.
\item There exists $j_t$, an almost complex structure on $\xi$ such that 
$J_t|_{\xi(z,r)} = j_t$ for $(z,r) \in \partial X \times [r_0,\infty)$.
\end{itemize}
We denote $\mca{J}(X,\lambda):=\bigcup_{r_0 \ge 1} \mca{J}(X,\lambda:r_0)$.

\begin{rem}
Recall that an almost complex structure $J$ on $\hat{X}$ is compatible with $d\hat{\lambda}$ if and only if 
the bilinear form on $T\hat{X}$
\[
\langle v, w \rangle_J:= d\hat{\lambda}(v,Jw)\quad \bigl(v,w \in T\hat{X})
\]
is a Riemannian metric.
Let us denote $\langle v,v\rangle_J^{1/2}$ by $|v|_J$.
\end{rem}

Let $H \in \mca{H}_{\ad}(X,\lambda)$, $J \in \mca{J}(X,\lambda)$.
For $x_-, x_+ \in \mca{P}(H)$, we consider the Floer equation for $u: \R \times \R/\Z \to \hat{X}$, namely:
\[
\partial_s u- J_t\bigl(\partial_t u - X_{H_t}(u)\bigr) =0, \qquad
u(s) \to x_{\pm}\,(s \to \pm \infty).
\]
In the second formula, $u(s)$ denotes the map $\R/\Z \to \hat{X}; t \mapsto u(s,t)$.
Let us denote the moduli space of solutions of the above Floer equations by $\hat{\mca{M}}\bigl(x_-,x_+:H,J\bigr)$.
$\hat{\mca{M}}\bigl(x_-,x_+:H,J\bigr)$ admits a natural $\R$ action:
\[
s_0 \cdot u(s,t):= u(s-s_0,t).
\]
$\mca{M}\bigl(x_-,x_+:H,J\bigr)$
denotes the quotient of $\hat{\mca{M}}\bigl(x_-,x_+:H,J\bigr)$ by the above $\R$ action.

For generic $J$, $\mca{M}\bigl(x_-,x_+:H,J\bigr)$ is a smooth manifold with dimension $\mu_{\CZ}(x_-) - \mu_{\CZ}(x_+)-1$.
For such $J$, we define the differential $\partial_{H,J}:C_k(H) \to C_{k-1}(H)$ by 
\[
\partial_{H,J}[x_-] := \sum_{x_+ \in \mca{P}_{k-1}(H)} \sharp \mca{M}\bigl(x_-,x_+:H,J\bigr)\cdot [x_+]\qquad \bigl(x_-\in \mca{P}_k(H) \bigr).
\]
Then, $\bigl(C_*(H), \partial_{H,J})$ becomes a chain complex.
It follows from
the following $C^0$ bound for Floer trajectories (this is a special case of lemma \ref{lem:C0bound-2}, 
which is stated later):

\begin{lem}\label{lem:C0bound-1}
There exists a compact set $B \subset \hat{X}$ such that for any
$x_-, x_+ \in \mca{P}(H)$ and 
$u \in \hat{\mca{M}}(x_-,x_+:H,J)$, $u(\R \times \R/\Z) \subset B$.
\end{lem}

It can be shown that the homology group of the complex $\bigl(C_*(H), \partial_{H,J}\bigr)$ is independent of 
choices of $J$, and we denote it by $\HF_*(H)$, or $\HF_*\bigl(H:(X,\lambda)\bigr)$, when we need to 
specify the Liouville domain.

Let $H_-, H_+ \in \mca{H}_{\ad}(X,\lambda)$ and assume that $a_{H_-} \le a_{H_+}$.
Then, there exists a canonical morphism $\HF_*(H_-) \to \HF_*(H_+)$. 
This is constructed as follows:
take $r_0 \ge 1$ and 
$(H_s)_{s \in \R}$, a family of elements in $\mca{H}(X,\lambda:r_0)$ and 
$(J_s)_{s \in \R}$, a family of elements in $\mca{J}(X,\lambda:r_0)$ which 
satisfy the following conditions:
\begin{itemize}
\item There exists $s_0>0$ such that $H_s = \begin{cases} 
                                             H_- &( s \le -s_0) \\
                                             H_+ &( s \ge s_0) 
                                            \end{cases}$,
$J_s = \begin{cases}
        J_{-s_0} &(s \le -s_0) \\
        J_{s_0}  &(s \ge s_0)
       \end{cases}$.
\item $\partial_s a_{H_s} \ge 0$.
\end{itemize}
For $x_- \in \mca{P}(H_-)$ and $x_+ \in \mca{P}(H_+)$, consider the Floer equation 
for $u: \R \times \R/\Z \to \hat{X}$:
\[
\partial_su - J_{s,t}\bigl(\partial_t u - X_{H_{s,t}}(u) \bigr)=0, \qquad
u(s)\to x_{\pm}\,(s \to \pm \infty), 
\]
where $H_{s,t}:=(H_s)_t$, $J_{s,t}:=(J_s)_t$.

We denote the moduli space of solutions of the above Floer equation by $\hat{\mca{M}}\bigl(x_-,x_+:(H_s,J_s)_s\bigr)$.
For generic $(J_s)_s$, $\hat{\mca{M}}\bigl(x_-,x_+:(H_s,J_s)_s\bigr)$ is a smooth manifold of dimension 
$\mu_{\CZ}(x_-) - \mu_{\CZ}(x_+)$. Taking such $(J_s)_s$, we define a morphism 
$\varphi \colon C_k(H_-) \to C_k(H_+)$ by 
\[
\varphi [x_-] = \sum_{x_+ \in \mca{P}_k(H_+)} \sharp \hat{\mca{M}}\bigl(x_-,x_+: (H_s,J_s)_s\bigr) \cdot [x_+]
\qquad \bigr(x_- \in \mca{P}_k(H_-)\bigr).
\]
Then, $\varphi$ is a chain map from $\bigl(C_*(H_-),\partial_{H_-,J_-}\bigr)$ to 
$\bigl(C_*(H_+), \partial_{H_+,J_+}\bigr)$.
It follows from the following $C^0$ bound for Floer trajectories
(it follows from lemma 1.5 in \cite{Oancea}):

\begin{lem}\label{lem:C0bound-2}
There exists a compact set $B \subset \hat{X}$ such that
for any $x_- \in \mca{P}(H_-)$, $x_+ \in \mca{P}(H_+)$ and 
$u \in \hat{\mca{M}}\bigl(x_-, x_+:(H_s,J_s)_s \bigr)$, $u(\R \times \R/\Z) \subset B$.
\end{lem}

Therefore, $\varphi$ defines a morphism $\varphi_* \colon \HF_*(H_-) \to \HF_*(H_+)$.
This morphism does not depend on choice of $(H_s,J_s)_s$.

To sum up, we have constructed the canonical morphism $\HF_*(H_-) \to \HF_*(H_+)$ 
for $H_-,H_+ \in \mca{H}_\ad(X,\lambda)$ such that $a_{H_-} \le a_{H_+}$.
This morphism is called \textit{monotone morphism}.

We also study truncated version of the Floer homology.
For any $x:\R/\Z \to \hat{X}$, let
\[
\mca{A}_H(x):=\int_{\R/\Z} x^*\hat{\lambda} - H\bigl(x(t)\bigr) dt.
\]
For any interval $I \subset [-\infty, \infty]$, 
let $C^I_k(H)$ be the free $\Z_2$-module generated over
\[
\bigl\{ x \in \mca{P}_k(H) \bigm{|} \mca{A}_H(x) \in I \bigr\}.
\]
For $x_-, x_+ \in \mca{P}(H)$ and $u \in \hat{\mca{M}}(x_-, x_+:H,J)$, by straightforward calculations we get
\[
-\partial_s\bigl(\mca{A}_H\bigl(u(s)\bigr)\bigr)= \int_{\R/\Z} \big|\partial_s u(s,t)\big|_{J_t}^2 dt.
\]

In particular, if $\mca{A}_H(x)<\mca{A}_H(y)$, then 
$\hat{\mca{M}}(x,y:H,J)=\emptyset$. 
Hence for any interval $I \subset [-\infty,\infty]$, 
$\bigl(C_*^I(H), \partial_{H,J}\bigr)$ is a chain complex.
Then, we denote $H_*\bigl(C_*^{I}(H),\partial_{H,J}\bigr)$ (which does not depend on $J$)
by $\HF^{I}_*(H)$.

For $-\infty \le a < b < c \le \infty$, there exists a short exact sequence
\[
0 \to C_*^{[a,b)}(H) \to C_*^{[a,c)}(H) \to C_*^{[b,c)}(H) \to 0.
\]
Hence we get a long exact sequence 
\begin{equation}\label{eq:abc-1}
\cdots \to \HF_*^{[a,b)}(H) \to \HF_*^{[a,c)}(H) \to \HF_*^{[b,c)}(H) \to  \HF_{*-1}^{[a,b)}(H) \to \cdots.
\end{equation}

\subsection{Symplectic homology}
Let $(X,\lambda)$ be a Liouville domain.
In this subsection, we define \textit{symplectic homology}
$\SH_*^I(X,\lambda)$ for any interval $I \subset \R$.

First, we define $\mca{H}^{\rest}(X,\lambda) \subset \mca{H}(X,\lambda)$ by
\[
\mca{H}^{\rest}(X,\lambda):=\bigl\{ H \in \mca{H}(X,\lambda:1) \bigm{|} \text{$H_t|_X<0$ for any $t \in \R/\Z$} \bigr\}.
\]
For $H_-, H_+ \in \mca{H}^{\rest}(X,\lambda)$, we denote $H_- \le H_+$ if and only if 
$(H_-)_t \le (H_+)_t$ for any $t \in \R/\Z$.

Let $H_-, H_+ \in \mca{H}^\rest_\ad(X,\lambda):=\mca{H}_\ad(X,\lambda) \cap \mca{H}^\rest(X,\lambda)$.
When $H_- \le H_+$, we can construct a morphism 
$\HF_*^I(H_-) \to \HF_*^I(H_+)$ for any interval $I \subset \R$.
This is constructed as follows.
First, take $(H_s)_{s \in \R}$, a family of elements of $\mca{H}^{\rest}(X,\lambda)$ and 
$(J_s)_{s \in \R}$, a family of elements of $\mca{J}(X,\lambda:1)$ which satisfy the following properties:
\begin{itemize}
\item There exists $s_0>0$ such that $H_s=\begin{cases}
                                           H_- &(s \le -s_0) \\
                                           H_+ &(s \ge s_0)
                                          \end{cases}$, 
                                      $J_s=\begin{cases}
                                           J_{-s_0} &(s \le -s_0) \\
                                           J_{s_0} &(s \ge s_0)
                                           \end{cases}$.
\item $\partial_sH_{s,t}(x) \ge 0$ for any $(s,t) \in \R \times \R/\Z$ and $x \in \hat{X}$.
\end{itemize}
For $x_- \in \mca{P}(H_-)$, $x_+ \in \mca{P}(H_+)$ and $u \in \hat{\mca{M}}\bigl(x_-,x_+:(H_s,J_s)_s\bigr)$, 
\[
-\partial_s\bigl(\mca{A}_{H_s}(u(s))\bigr)
=\int_{\R/\Z} |\partial_s u|_{J_{s,t}}^2 + \partial_sH_{s,t}(u) dt \ge 0.
\]
Hence if $\mca{A}_{H_-}(x_-)<\mca{A}_{H_+}(x_+)$, then 
$\hat{\mca{M}}\bigl(x_-,x_+:(H_s,J_s)_s\bigr)=\emptyset$.

Therefore the morphism 
$\varphi: \bigl(C_*^I(H_-),\partial_{H_-,J_-}\bigr) \to \bigl(C_*^I(H_+), \partial_{H_+,J_+}\bigr)$ 
defined by 
\[
\varphi [x_-] = \sum_{x_+ \in \mca{P}_k(H_+)} \sharp \hat{\mca{M}}\bigl(x_-,x_+:(H_s,J_s)_s\bigr) \cdot [x_+]
\qquad \bigr(x_- \in \mca{P}_k(H_-)\bigr)
\]
is a chain map.
Hence we get a morphism $\HF_*^I(H_-) \to \HF_*^I(H_+)$.
This morphism does not depend on choices of $(H_s,J_s)_s$.
Then, we define $\SH_*^I(X,\lambda)$ by 
\[
\SH_*^I(X,\lambda):= \varinjlim_{H \in \mca{H}^\rest_\ad(X,\lambda)} \HF_*^I(H).
\]

For $-\infty \le a<b<c \le \infty$, by taking limit of (\ref{eq:abc-1}), we get a long exact sequence
\begin{equation}\label{eq:abc-2}
\cdots \to \SH_*^{[a,b)}(X,\lambda) \to \SH_*^{[a,c)}(X,\lambda) \to \SH_*^{[b,c)}(X,\lambda) \to \SH_{*-1}^{[a,b)}(X,\lambda) \to \cdots.
\end{equation}

For $a \in (-\infty,\infty]$, $\SH_*^{(-\infty,a)}(X,\lambda)$ is often denoted by $\SH_*^{<a}(X,\lambda)$.
$\SH_*^{<\infty}(X,\lambda)$ is often abbreviated by $\SH_*(X,\lambda)$. 
The following lemma will be useful in later:

\begin{lem}\label{lem:firstdef}
For any $H \in \mca{H}_\ad(X,\lambda)$, there exists a canonical isomorphism
$\SH_*^{<a_H}(X,\lambda) \to \HF_*(H)$.
When $H_-, H_+ \in \mca{H}_\ad(X,\lambda)$ satisfy $a_{H_-} \le a_{H_+}$, the following diagram commutes:
\[
\xymatrix{
\SH_*^{<a_{H_-}}(X,\lambda) \ar[r]^-{\cong}\ar[d] & \HF_*(H_-)\ar[d]\\
\SH_*^{<a_{H_+}}(X,\lambda) \ar[r]_-{\cong}& \HF_*(H_+)
}.
\]
\end{lem}
\begin{proof}
It is not hard to check that the following natural morphisms are all isomorphic:
\begin{align*}
\varinjlim_{\substack{G \in \mca{H}^\rest_\ad \\ a_G \le a_H}} \HF_*^{<a_H}(G) &\to 
\varinjlim_{G \in \mca{H}^\rest_\ad} \HF_*^{<a_H}(G)=\SH_*^{<a_H}(X,\lambda), \\
\varinjlim_{\substack{G \in \mca{H}^\rest_\ad \\ a_G \le a_H}} \HF_*^{<a_H}(G) &\to
\varinjlim_{\substack{G \in \mca{H}^\rest_\ad \\ a_G \le a_H}} \HF_* (G), \\
\varinjlim_{\substack{G \in \mca{H}^\rest_\ad \\ a_G \le a_H}} \HF_*(G) &\to 
\varinjlim_{\substack{G \in \mca{H}_\ad \\ a_G \le a_H}} \HF_*(G), \\
\varinjlim_{\substack{G \in \mca{H}_\ad \\ a_G \le a_H}} \HF_*(G)& \to \HF_*(H).
\end{align*}
By composing the above isomorphisms and their inverses, we get an isomorphism
$\SH_*^{<a_H}(X,\lambda) \to \HF_*(H)$. This proves the first assertion.
The second assertion follows from the above construction.
\end{proof}

We recall three well-known results on symplectic homology.
All these results were established in \cite{Viterbo}.
The first result is the following:

\begin{thm}\label{thm:delta}
For any $0< \delta \le \delta(\partial X,\lambda)$, 
there exists a canonical isomorphism $\SH_*^{<\delta}(X,\lambda) \to H_{*+n}(X,\partial X)$.
\end{thm}

The second result is the following:

\begin{thm}\label{thm:isotopyinvariance}
If $(X,\lambda)$ and $(Y,\lambda')$ are isotopic as Liouville domains, then 
$\SH_*(X,\lambda) \cong \SH_*(Y,\lambda')$.
\end{thm}

As a corollary of the above theorem, we can conclude that $\SH_*(X,\lambda)$ depends only on $d\lambda$.
Assume that $(X,\lambda), (X,\lambda')$ are Liouville domains, and $d\lambda=d\lambda'$.
Then, $\bigl(X, t\lambda + (1-t)\lambda')_{0 \le t \le 1}$ is a family of Liouville domains, and theorem \ref{thm:isotopyinvariance}
implies that $\SH_*(X,\lambda) \cong \SH_*(X,\lambda')$.
Hence we often denote $\SH_*(X,\lambda)$ by $\SH_*(X,d\lambda)$.

The third result is the following:

\begin{thm}\label{thm:ball}
For positive integer $n$ and $r>0$, let 
\[
B^{2n}(r):=\bigl\{(q,p) \in \R^{2n} \bigm{|} |q|^2+|p|^2 \le r^2 \bigr\}.
\]
Then, $\bigl(B^{2n}(r),\lambda_n\bigr)$ is a Liouville domain, and 
$\SH_*\bigl(B^{2n}(r),\lambda_n\bigr)=0$.
\end{thm}

For proofs, see proposition 1.4 in \cite{Viterbo} for theorem \ref{thm:delta}, 
theorem 1.7 in \cite{Viterbo} for theorem \ref{thm:isotopyinvariance}, and 
section 4, example 1 in \cite{Viterbo} for theorem \ref{thm:ball}.

Note that theorem \ref{thm:isotopyinvariance} does not hold for truncated symplectic homology.
However, the following result holds:

\begin{thm}\label{thm:truncated}
Let $(X,\lambda)$, $(X,\lambda')$ be Liouville domains, and assume that $d\lambda=d\lambda'$.
Then, for any $a \in (0,\infty]$, there exists a canonical isomorphism
$\psi^{<a}: \SH_*^{<a}(X,\lambda) \to \SH_*^{<a}(X,\lambda')$.
Moreover, for any $0<a \le b \le \infty$, 
\[
\xymatrix{
\SH_*^{<a}(X,\lambda) \ar[r]^{\psi^{<a}}\ar[d] & \SH_*^{<a}(X,\lambda')\ar[d] \\
\SH_*^{<b}(X,\lambda) \ar[r]_{\psi^{<b}}& \SH_*^{<b}(X,\lambda')
}
\]
commutes.
\end{thm}

Theorem \ref{thm:truncated} is proved in the appendix.

\section{Symplectic capacity via symplectic homology}\label{sec:caps}

\begin{defn}\label{defn:capp}
Let $(X,\lambda)$ be a Liouville domain.
$\capp_S(X,\lambda)$ is defined by 
\[
\capp_S(X,\lambda):= \inf\bigl\{ a \in (0,\infty] \bigm{|} 
\text{$\SH_n^{<\delta(\partial X,\lambda)}(X,\lambda) \to \SH_n^{<a}(X,\lambda)$ vanishes} \bigr\}.
\]
\end{defn}

\begin{rem}\label{rem:viterbo}
The above capacity is introduced by C.Viterbo in \cite{Viterbo}, section 5.3.
\end{rem}

\begin{lem}\label{lem:capacity}
Let $(X,\lambda)$ be a Liouville domain.
\begin{enumerate}
\item[(1)] $\capp_S(X,a\lambda) = a  \cdot \capp_S(X,\lambda)$ for any $a \in (0,\infty)$.
\item[(2)] $\SH_*(X,\lambda)=0 \implies \capp_S(X,\lambda) < \infty$.
\item[(3)] $\capp_S(X,\lambda)$ depends only on $d\lambda$.
\end{enumerate}
\end{lem}
\begin{proof}
(1) and (2) are immediate from the definition.
(3) follows from theorem \ref{thm:truncated}.
\end{proof}

The goal of this section is to prove the following three properties of $\capp_S$.

\begin{thm}\label{thm:monotonicity}
Let $(X,\lambda)$ be a Liouville domain, and
$X_{\inn}$ be a submanifold of $X$ of codimension $0$.
If $(X_{\inn}, \lambda)$ is a Liouville domain, then 
$\capp_S(X_{\inn}, \lambda) \le \capp_S (X,\lambda)$.
\end{thm}

\begin{thm}\label{thm:reebchord}
Let $(X,\lambda)$ be a Liouville domain.
If $\capp_S(X,\lambda)<\infty$, there exists $x \in \mca{P}_0(\partial X,\lambda)$ such that 
$\tau(x) = \capp_S(X,\lambda)$, 
$\mu_{\CZ}(x) \le n+1$.
\end{thm}

\begin{thm}\label{thm:covering}
Let $\pi \colon Y \to X$ be a covering map such that $\deg \pi < \infty$.
If $(X,\lambda)$ is a Liouville domain, then 
$(Y,\pi^*\lambda)$ is a Liouville domain, and 
$\capp_S(Y,\pi^*\lambda) \le \capp_S(X,\lambda)$.
\end{thm}

\begin{rem}
Conley-Zehnder index for elements in $\mca{P}_0(\partial X,\lambda)$, which appears on the statement of 
theorem \ref{thm:reebchord} have not been defined. It is defined at the beginning of section 3.3.
\end{rem}

It seems to the author that various variants of results semilar to theorems \ref{thm:monotonicity},
\ref{thm:reebchord} are known or expected to hold by experts.
We give its proof below for the sake of completeness since the author is unable to find 
their proofs in the literature.
On the other hand, theorem \ref{thm:covering} is new, though its proof is not very difficult.
Theorem \ref{thm:covering} plays a crucial role in the proof of theorem \ref{thm:mainestimate}, which is 
the main result in section 4.

\subsection{Proof of theorem \ref{thm:monotonicity}}
First we prove the following lemma.

\begin{lem}\label{lem:convexbdry}
Let $(X,\lambda)$ and $X_{\inn}$ be as in theorem \ref{thm:monotonicity}, 
and $\varepsilon \in (0,1)$.
Let $H_-, H_+ \in \mca{H}_\ad(X,\lambda)$, 
$(H_s)_{s \in \R}$ be a family of elements of $\mca{H}(X,\lambda)$, 
$(J_s)_{s \in \R}$ be a family of elements of $\mca{J}(X,\lambda)$.
Assume that they satisfy the following conditions:
\begin{enumerate}
\item[(i)] There exists $s_0>0$ such that $H_s = \begin{cases}
                                                H_- &(s \le -s_0) \\
                                                H_+ &(s \ge s_0) 
                                                \end{cases}$.
\item[(ii)] $\partial_sH_{s,t}(x) \ge 0$ for any $(s,t) \in \R \times \R/\Z$ and $x \in \hat{X}$.
\item[(iii)]There exists $a \in C^\infty(\R)$ such that 
$H_{s,t}(z,r)=a(s)(r-\varepsilon)$ for $(z,r) \in \partial X_{\inn} \times [\varepsilon^{2/3}, \varepsilon^{1/3}]$.
\item[(iv)] $dr \circ J_{s,t}=-\lambda$ on $\partial X_{\inn} \times [\varepsilon^{2/3},\varepsilon^{1/3}]$.
\end{enumerate}
Assume $x_- \in \mca{P}(H_-)$ and $x_+ \in \mca{P}(H_+)$ satisfy
$x_-(\R/\Z), x_+(\R/\Z) \subset X_\inn(\varepsilon^{1/3})$.
Then, for any $u \in \hat{\mca{M}}\bigl(x_-, x_+:(H_s,J_s)_s\bigr)$, 
$u(\R \times \R/\Z) \subset X_\inn(\varepsilon^{1/3})$.
\end{lem}
\begin{proof}
First notice that $x_-(\R/\Z), x_+(\R/\Z)$ are contained in $X_{\inn}(\varepsilon^{2/3})$, since 
$(H_\pm)_t(z,r)=a(\pm s_0)(r-\varepsilon)$ for 
$(z,r) \in \partial X_{\inn} \times [\varepsilon^{2/3},\varepsilon^{1/3}]$ and
$a(\pm s_0) \notin \mca{\tau}(\partial X_{\inn}, \lambda)$
(this follows from (i), (iii) and $H_\pm \in \mca{H}_\ad(X,\lambda)$).
We will prove that for any $u \in \hat{\mca{M}}\bigl(x_-, x_+:(H_s,J_s)_s\bigr)$, 
$u(\R \times \R/\Z) \subset X_{\inn}(\varepsilon^{1/3})$.
If this is not true, for any $r_0 \in (\varepsilon^{2/3},\varepsilon^{1/3})$,
$D_{r_0}:=\R \times \R/\Z \setminus u^{-1}\bigl(\interior X_{\inn}(r_0)\bigr)$ is  non-empty.
Note that $D_{r_0}$ is compact since $x_\pm(\R/\Z) \subset X_{\inn}(\varepsilon^{2/3})$.
For generic $r_0$, $u$ is transverse to $\partial X_{\inn} \times \{r_0\}$, hence we may assume that 
$D_{r_0}$ is a compact surface with boundary.

It is easy to verify that $\partial_s u$ is not constantly $0$ on $D_{r_0}$. Hence
\[
\int_{D_{r_0}} |\partial_s u|_{J_{s,t}}^2\,dsdt >0.
\]
Since $u$ satisfies the Floer equation $\partial_s u - J_{s,t}\bigl(\partial_t u - X_{H_{s,t}}(u)\bigr)=0$, 
\begin{align*}
&\int_{D_{r_0}} |\partial_s u|_{J_{s,t}}^2 + \partial_sH_{s,t}(u)\,dsdt = 
\int_{D_{r_0}} d\hat{\lambda}(\partial_t u, \partial_s u) + dH_{s,t}(\partial_s u) + \partial_sH_{s,t}(u)\,dsdt \\
&=\int_{\partial D_{r_0}} -u^*\lambda + H_{s,t}(u)\,dt.
\end{align*}
Since $u(\partial D_{r_0}) \subset \partial X_{\inn} \times \{r_0\}$, we get by (iii)
\[
(s,t) \in \partial D_{r_0} \implies H_{s,t}\bigl(u(s,t)\bigr)=a(s)(r_0-\varepsilon), \quad \lambda\bigl(X_{H_{s,t}}(u(s,t))\bigr)=a(s)r_0.
\]
Therefore
\[
\int_{\partial D_{r_0}} -u^*\lambda + H_{s,t}(u)\,dt 
=\int_{\partial D_{r_0}} \lambda(X_{H_{s,t}} \otimes dt - du) - \varepsilon \int_{\partial D_{r_0}}a(s)\,dt.
\]
On the other hand, the Floer equation is equivalent to 
\[
J_{s,t} \circ (X_{H_{s,t}} \otimes dt - du) = -(X_{H_{s,t}} \otimes dt - du) \circ j,
\]
where $j$ is the complex structure on $\R \times \R/\Z$, defined by $j(\partial_s)=\partial_t$.
Therefore by (iv), 
\[
\int_{\partial D_{r_0}} \lambda(X_{H_{s,t}} \otimes dt - du) 
=\int_{\partial D_{r_0}} \lambda\bigl(J_{s,t} \circ (X_{H_{s,t}} \otimes dt - du) \circ j\bigr)
=\int_{\partial D_{r_0}} dr(X_{H_{s,t}} \otimes dt - du) \circ j .
\]
$dr(X_{H_{s,t}})=0$ on $\partial X_{\inn} \times \{r_0\}$.
Moreover, if $V$ is a vector tangent to $\partial D_{r_0}$, and positive with respect to the boundary orientation, 
then $jV$ points inwards, hence $dr\bigl(du(jV)\bigr) \ge 0$. Therefore,
\[
\int_{\partial D_{r_0}} \lambda(X_{H_{s,t}} \otimes dt - du) \le 0.
\]
Finally, 
\[
\int_{D_{r_0}} |\partial_s u|_{J_{s,t}}^2 + \partial_s H_{s,t}(u)\,dsdt
\le -\varepsilon \int_{\partial D_{r_0}} a(s)\,dt 
= -\varepsilon \int_{D_{r_0}} \partial_s a(s)\,dsdt.
\]
Since $\partial_s H_{s,t} \ge 0$ and $\partial_s a \ge 0$ by (ii), 
this implies
\[
\int_{D_{r_0}} |\partial_s u|_{J_{s,t}}^2 dsdt \le 0.
\]
This is a contradiction.
\end{proof}

We prove theorem \ref{thm:monotonicity}.

\begin{proof}
We prove that, if $a$ satisfies 
$a>\capp_S(X,\lambda)$ and $a \notin \tau(\partial X,\lambda) \cup \tau(\partial X_{\inn},\lambda)$, 
then $a>\capp_S(X_{\inn},\lambda)$.
This implies $\capp_S(X,\lambda) \ge \capp_S(X_{\inn},\lambda)$, 
since $\tau(\partial X,\lambda)$ and $\tau(\partial X_\inn,\lambda)$ are
null sets.
In the rest of this proof, we assume that $X$ and $X_\inn$ are connected. 
General case follows at once from this particular case.

Take $\varepsilon>0$ so that $\bigl[a(1-\varepsilon), a\bigr]$ is disjoint from $\tau(\partial X,\lambda)$ and $\tau(\partial X_{\inn},\lambda)$.
For any $c>0$, define $H_c: \hat{X_{\inn}} \to \R$ and $K_c: \hat{X} \to \R$ as follows:
\begin{align*}
H_c(x) &=\begin{cases}
          0         &\bigl(x \in X_{\inn}(\varepsilon)\bigr) \\
         c(r-\varepsilon)&\bigl(x=(z,r) \in \partial X_{\inn} \times [\varepsilon,\infty)\bigr) \\
         \end{cases}, \\
K_c(x) &= \begin{cases}
          H_c(x)    &(x \in X_{\inn}) \\
        c(1-\varepsilon)    &(x \in X \setminus X_{\inn}) \\
        c(r-\varepsilon) &\bigl(x=(z,r) \in \partial X \times [1,\infty) \bigr).
         \end{cases}
\end{align*}

Take $\delta>0$ so small that $\delta< \min\bigl\{ \delta(\partial X,\lambda), \delta(\partial X_{\inn},\lambda)\bigr\}$.
Then, perturbing $K_a,K_\delta$ and $H_a,H_\delta$ respectively, we can take 
$K'_a, K'_\delta \in \mca{H}_{\ad}(X,\lambda)$ and $H'_a, H'_\delta \in \mca{H}_{\ad}(X_{\inn},\lambda)$ which satisfy the
following properties:
\begin{enumerate}
\item[(i)] For $c \in \{\delta,a\}$, the following holds:
\begin{enumerate}
\item $(H'_c)_t=H_c$ on $\partial X_{\inn} \times [\varepsilon^{2/3},\infty)$ for any $t \in \R/\Z$.
\item $(K'_c)_t=K_c$ on $\partial X \times [2,\infty)$ for any $t \in \R/\Z$.
\item $(K'_c)_t=(H'_c)_t$ on $X_{\inn}(\varepsilon^{1/3})$ for any $t \in \R/\Z$.
\item For $x \in \mca{P}(K'_c)$, $\mca{A}_{K'_c}(x)>0$ if and only if $x(\R/\Z) \subset X_{\inn}(\varepsilon^{1/3})$.
\end{enumerate}
\item[(ii)] $(K'_\delta)_t \le (K'_a)_t$ and $(H'_\delta)_t \le (H'_a)_t$ for any $t \in \R/\Z$.
\item[(iii)] $H'_\delta$ and $K'_\delta$ are time independent. i.e. There exist $h \in C^\infty(\hat{X_\inn})$ and $k \in C^\infty(\hat{X})$ 
such that $(H'_\delta)_t=h$, $(K'_\delta)_t=k$. Moreover, 
$\mca{P}(H'_\delta) = \Crit(h)$, $\mca{P}(K'_\delta)=\Crit(k)$ and 
if $p \in \Crit (k)$ satisfies $\ind\,p=0$, then $p \in X_\inn$.
\end{enumerate}

Let $c \in \{\delta,a\}$.
Then, by (i)-(c) and (i)-(d), if $x \in \mca{P}(K'_c)$ satisfies $\mca{A}_{K'_c}(x)>0$, $x$ can be identified with a 
solution of $\partial_t x = X_{(H'_c)_t}(x)$.
We define $\psi_c: C^{>0}_*(K'_c) \to C_*(H'_c)$ by 
\[
\psi_c[x] = \begin{cases}
             [x] &(\text{$x$ is contractible in $X_{\inn}$}) \\
              0  &(\text{otherwise})
             \end{cases}.
\]
'if' part of (i)-(d) implies that $\psi_c$ is an epimorphism.

Let $J=(J_t)_{t \in \R/\Z}$ be a family of almost complex structures on $X_{\inn}$, such that 
each $J_t$ is compatible with $d\lambda$ and satisfies $dr \circ J_t = -\lambda$ on $X_{\inn} \times [\varepsilon^{2/3},\varepsilon^{1/3}]$.

By (i)-(a) and (i)-(c), $(K'_c)_t(z,r)=c(r-\varepsilon)$ for $(z,r) \in X_{\inn} \times [\varepsilon^{2/3},\varepsilon^{1/3}]$.
Therefore, by lemma \ref{lem:convexbdry}, if we extend $J$ to $J^X \in \mca{J}(X,\lambda)$ and
$J^{X_{\inn}} \in \mca{J}(X_{\inn},\lambda)$, 
$\psi_c$ defines a chain map from 
$\bigl(C_*^{>0}(K'_c),\partial_{K'_c,J^X}\bigr)$ to
$\bigl(C_*(H'_c),\partial_{H'_c,J^{X_{\inn}}}\bigr)$.

It induces a morphism
\[
\HF_*(K'_c) \to \HF^{>0}_*(K'_c) \to \HF_*(H'_c).
\]
We will denote this morphism also by $\psi_c$.

It follows from lemma \ref{lem:convexbdry} and (ii) that
\begin{equation}\label{eq:monotonicity}
\xymatrix{
\HF_*(K'_\delta)\ar[r]^{\psi_\delta}\ar[d]&\HF_*(H'_\delta)\ar[d]\\
\HF_*(K'_a)\ar[r]_{\psi_a}&\HF_*(H'_a)
}
\end{equation}
commutes, where vertical morphisms are monotone morphisms.

We complete the proof.
We have to show that
if $\SH_n^{<\delta}(X,\lambda) \to \SH_n^{<a}(X,\lambda)$ vanishes, then
$\SH_n^{<\delta}(X_\inn,\lambda) \to \SH_n^{<a}(X_\inn,\lambda)$ vanishes.

By (i)-(b), 
$a_{K'_c}=a_{H'_c}=c$ for $c \in \{\delta,a\}$.
Hence by lemma \ref{lem:firstdef}, it is enough to prove that
if $\HF_n(K'_\delta) \to \HF_n(K'_a)$ vanishes, then
$\HF_n(H'_\delta) \to \HF_n(H'_a)$ vanishes.

By (iii), 
$C_k(H'_\delta) = C_k(K'_\delta)=0$ for $k \ge n+1$, 
and $C_n(K'_\delta)$ is identified with $C_n(H'_\delta)$.
Hence $\psi_\delta: \HF_n(K'_\delta) \to \HF_n(H'_\delta)$ is injective, therefore isomorphic
(recall that we have assumed $X$ and $X_\inn$ to be connected).
Then, (\ref{eq:monotonicity}) implies that 
if $\HF_n(K'_\delta) \to \HF_n(K'_a)$ vanishes, then
$\HF_n(H'_\delta) \to \HF_n(H'_a)$ vanishes.
\end{proof}

\subsection{Proof of theorem \ref{thm:reebchord}}

First we define the Conley-Zehnder index for elements of $\mca{P}_0(\partial X, \lambda)$.
We assume that $n \ge 2$.
Let $x \in \mca{P}_0(\partial X, \lambda)$.
Then, there exists $\bar{x}: D^2 \to X$ such that 
$\bar{x}(e^{2\pi i\theta}) = x(\tau\theta)$.
Take a trivialization of $\bar{x}^*TX$ as symplectic vector bundle, 
\[
F:(\R^{2n},\omega_n) \times D^2 \to \bar{x}^*TX; \qquad
(v,z) \mapsto \bigl(F_z(v),z\bigr),
\]
such that for any $\theta \in \R/\Z$, the following holds:
\begin{equation}\label{eq:CZ}
\begin{cases}
&F_{e^{2\pi i\theta}}(0,\ldots,0,0,1)=\partial_r\bigl(x(\tau\theta)\bigr), \\
&F_{e^{2\pi i\theta}}(0,\ldots,0,1,0)=R\bigl(x(\tau\theta)\bigr), \\
&F_{e^{2\pi i\theta}}\bigl({\R}^{2n-2}\times(0,0)\bigr)=\xi\bigl(x(\tau\theta)\bigr).
\end{cases}
\end{equation}
Note that such trivialization exists only if $n \ge 2$.

Define a symplectic path $\gamma \in \mca{P}_\tau(2n-2)$ by 
\[
\gamma(t):= \bigl(F_{e^{2\pi it/\tau}}|_{\R^{2n-2}\times (0,0)}\bigr)^{-1} \circ \Phi_t|_\xi \circ F_1|_{\R^{2n-2}\times (0,0)},
\]
where $(\Phi_t)_t$ is the Poincar\'{e} map of the flow generated by $R$ on $\partial X$. 
Then, define 
\[
\mu_\CZ(x):=i(\gamma).
\]
$x$ is called \textit{nondegenerate} if and only if $\gamma \in \mca{P}^*_\tau(2n-2)$.
The following lemma will be useful in later (note that it also implies that the above 
definition is consistent, i.e. it does not depend on choices of $\bar{x}$).

\begin{lem}\label{lem:CZ}
Let $H \in C^\infty(\hat{X})$ such that
$\partial X=H^{-1}(0)$ and $\partial_r H>0$ on $\partial X$.
Then, there exists $1:1$ correspondence between elements of $\mca{P}_0(\partial X,\lambda)$ and 
periodic orbits of $X_H$ on $\partial X$, which are contractible in $X$.
For $x \in \mca{P}_0(\partial X, \lambda)$, denote the corresponding periodic orbit of $X_H$ by $x_H$.
When $n \ge 2$, 
\[
\mu_{\CZ}(x) = \sup_H \mu_{\CZ}(x_H),
\]
where $H$ runs over all Hamiltonians satisfying the conditions as above. 
\end{lem}
\begin{proof}
The first assertion is obvious. 
We prove the second assertion.
Let $x \in \mca{P}_0(\partial X,\lambda)$, and $x_H$ be the corresponding periodic orbit of $X_H$ with period $\tau$.
Take $\bar{x}:D^2 \to X$ such that $\bar{x}(e^{2\pi i\theta})=x_H(\tau\theta)$, and 
take a trivialization of $\bar{x}^*TX$ as symplectic vector bundle 
$F:D^2 \times (\R^{2n},\omega_n) \to \bar{x}^*TX$, which satisfies (\ref{eq:CZ}).

Define $\Gamma \in \mca{P}_\tau(2n)$ by 
\[
\Gamma(t):=(F_{e^{2\pi it/\tau}})^{-1} \circ \Phi_t \circ F_1
\]
where $(\Phi_t)_t$ is the Poincar\'{e} map of the flow generated by $X_H$. Then, $\Gamma(t)$ can be written in the form 
\[
\Gamma(t) = \begin{pmatrix}
                  \gamma(t)&0&0\\
                  0&1&0\\
                  0&a(t)&1
                  \end{pmatrix}.
\]
Denote the symplectic path $t \mapsto \begin{pmatrix}1&0\\a(t)&1\end{pmatrix}$ by $\alpha$. Then, 
by theorem \ref{thm:CZ}-(2), 
$i(\Gamma)=i(\alpha)+i(\gamma)$. By definition, $i(\gamma)=\mu_{\CZ}(x)$. 
On the other hand, it is easy to verify that 
\[
i(\alpha ) =\begin{cases}
            -1 &\bigl(a(1) \le 0\bigr)\\
            0  &\bigl(a(1)  >  0 \bigr)
           \end{cases}.
\]
This proves the second assertion.
\end{proof}

By lemma \ref{lem:CZ}, 
it is possible to define the Conley-Zehnder index for $x \in \mca{P}_0(\partial X, \lambda)$ in 
another way, i.e.
\[
\mu_{\CZ}(x):= \sup_H \mu_{\CZ}(x_H),
\]
where $H$ runs over all elements in $C^\infty(\hat{X})$ such that $\partial X = H^{-1}(0)$ and 
$\partial_r H>0$ on $\partial X$. Note that this definition makes sense even when $n=1$.

\begin{cor}\label{cor:CZ}
Let $H \in C^\infty(\hat{X})$ such that $\partial X = H^{-1}(0)$ and $\partial_r H>0$ on $\partial X$.
Assume that there exist $0<r_0<1$ and $h \colon [r_0,\infty) \to \R$ such that
$H(z,r)=h(r)$ and $\partial_r^2h(1)>0$. 
Then, for any $x \in \mca{P}_0(\partial X, \lambda)$, 
$\mu_{\CZ}(x_H) = \mu_{\CZ}(x)$.
\end{cor}
\begin{proof}
First consider the case $n \ge 2$.
We use notations in the proof of lemma \ref{lem:CZ}.
Then, if $H$ satisfies the condition as the above statement, 
$a(1)>0$. Hence $\mu_{\CZ}(x_H)=\mu_{\CZ}(x)$.
The case $n=1$ is proved by similar arguments.
\end{proof}

In the rest of this subsection, we prove theorem \ref{thm:reebchord}.
First we consider cases in which all elements of $\mca{P}_0(\partial X, \lambda)$ are non-degenerate.

\begin{lem}\label{lem:reebchord}
Let $(X,\lambda)$ be as in theorem \ref{thm:reebchord}.
Assume that all elements in $\mca{P}_0(\partial X,\lambda)$ are non-degenerate.
Then, there exists $x \in \mca{P}_0(\partial X,\lambda)$ such that $\tau(x) = \capp_S(X,\lambda)$ and 
$\mu_{\CZ}(x) \in \{n,n+1\}$.
\end{lem}
\begin{proof}
We claim that for any $\varepsilon>0$, there exists
$x_\varepsilon \in \mca{P}_0(\partial X, \lambda)$ such that 
$\big\lvert \capp_S(X,\lambda) - \tau(x_\varepsilon) \big\rvert < \varepsilon$ 
and $\mu_{\CZ}(x_\varepsilon) \in \{n,n+1\}$.
Since all elements in $\mca{P}_0(\partial X,\lambda)$ are non-degenerate,
$\tau(\partial X, \lambda) \cap (0,T)$ is a finite set for any $T>0$.
Therefore, for sufficiently small $\varepsilon>0$, $\tau(x_\varepsilon)= \capp_S(X,\lambda)$.

We prove the above claim. It is enough to show the claim for sufficiently small $\varepsilon>0$.
In particular, we may assume that $\varepsilon/2 <\capp_S(X,\lambda)$.
The proof consists of $3$ steps.

\textbf{Step 1.}
First, take $(G^i)_i$, a sequence of time-independent Hamiltonians on $\hat{X}$ which satisfies the following
properties:
\begin{itemize}
\item $(G^i)_i$ is a cofinal sequence in $\mca{H}^\rest(X,\lambda)$, i.e. for any $G \in \mca{H}^\rest(X,\lambda)$,
$G_t \le G^i$ for any $t \in \R/\Z$ when $i$ is sufficiently large.
\item $G^i|_{X(1/2)}$ is sufficiently small in $C^2$ norm.
\item There exists $g^i:[1/2,\infty) \to \R$ such that $G^i(z,r)=g^i(r)$ on $\partial X \times [1/2,\infty)$ and 
$\partial_r^2g^i>0$ on $(1/2,1)$.
\end{itemize}

Then, $\mca{P}(G^i)$ consists of constant maps to $\Crit(G^i)$ and $S^1$-family of degenerate periodic orbits.
There exists a 1:1 correspondence between $S^1$-family of periodic orbits and 
elements of $\mca{P}_0(\partial X, \lambda)$ with periods less than $a_{G^i}$.
Let $x \in \mca{P}_0(\partial X, \lambda)$ such that $\tau(x)<a_{G^i}$, and 
let $\gamma_x$ be an element of a $S^1$-family of periodic orbits which corresponds to $x$.
Then, it follows from corollary \ref{cor:CZ} and $\partial_r^2g^i>0$ on $(1/2,1)$ that $\mu_{\CZ}(\gamma_x) = \mu_{\CZ}(x)$.
Moreover, by replacing $G^i$ if necessary, we may assume that
$\big\lvert \mca{A}_{G^i}(\gamma_x) - \tau(x) \big\rvert < \varepsilon/2$.

\textbf{Step 2.}
Perturbing each $(G^i)_i$, we can construct $(H^i)_i$, a sequence in $\mca{H}^\rest_\ad(X,\lambda)$ with 
the following properties:
\begin{enumerate}
\item[(i)] $(H^i)_i$ is a cofinal sequence in $\mca{H}^\rest_\ad(X,\lambda)$. i.e. for any $H \in \mca{H}^\rest_\ad(X,\lambda)$, 
$H_t \le H^i_t$ for any $t \in \R/\Z$ for sufficiently large $i$.
\item[(ii)] $H^i|_{X(1/2)}$ is time-independent, i.e. there exists 
$h^i \in C^\infty\bigl(X(1/2)\bigr)$ such that $H^i_t|_{X(1/2)}=h^i$ for any $t \in \R/\Z$.
\item[(iii)] For each $x \in \mca{P}_0(\partial X,\lambda)$ such that $\tau(x)<a_{H^i}$, 
there exists $x^{\pm} \in \mca{P}(H^i)$ such that 
$\mu_{CZ}(x^{\pm}) = \mu_{\CZ}(x)+(1\pm1)/2$, 
$\big\lvert \mca{A}_{H^i}(x^{\pm}) - \tau(x) \big\rvert < \varepsilon/2$.
\item[(iv)] $\mca{P}(H^i)$ consists of constant maps to $\Crit(h^i)$ and 
$\bigl\{ x^{\pm} \bigm{|} x \in \mca{P}_0(\partial X,\lambda), \tau(x)<a_{H^i}\bigr\}$.
\end{enumerate}
Precise arguments on perturbations are carried out as in \cite{CFHW}, proposition 2.2.

\textbf{Step 3.}
Abbreviate $\capp_S(X,\lambda)$ by $c$.
By definition of $\capp_S$, 
$\SH_n^{<c-\varepsilon/2}(X,\lambda) \to \SH_n^{<c+\varepsilon/2}(X,\lambda)$ is not injective.
Then $\SH_{n+1}^{[c-\varepsilon/2,c+\varepsilon/2)}(X,\lambda) \ne 0$, for the long exact sequence
\[
\cdots \to \SH_{n+1}^{[c-\varepsilon/2,c+\varepsilon/2)}(X,\lambda) \to 
\SH_n^{<c-\varepsilon/2}(X,\lambda) \to \SH_n^{<c+\varepsilon/2}(X,\lambda) \to \cdots.
\]
Therofore, by (i), $\HF_{n+1}^{[c-\varepsilon/2,c+\varepsilon/2)}(H^i) \ne 0$ for sufficiently large $i$.
This implies that there exists $x_i \in \mca{P}(H^i)$ such that 
$\mca{A}_{H^i}(x_i) \in [c-\varepsilon/2, c+\varepsilon/2)$ and $\mu_{\CZ}(x_i)=n+1$.
Since $(H^i)_i$ is cofinal in $\mca{H}_{\rest}(X,\lambda)$, 
we may assume that $\inf h^i> \varepsilon/2-c$.
Hence $x_i$ is not a constant map to $\Crit(h^i)$, and by (iv), 
there exists $x \in \mca{P}_0(\partial X,\lambda)$ such that
$x_i = x^+$ or $x_i= x^-$. 
By (iii), $\tau(x) \in [c-\varepsilon, c+\varepsilon)$ and 
$\mu_{\CZ}(x) \in \{n, n+1\}$.
Hence we have proved the claim.
\end{proof}

We prove theorem \ref{thm:reebchord}.

\begin{proof}
Let $(\hat{X},\hat{\lambda})$ be the completion of $(X,\lambda)$.
For any positive smooth function $f$ on $\partial X$, let $\Sigma_f$ be the hypersurface in $\hat{X}$ 
defined by $\bigl\{(z,f(z)) \bigm{|} z \in \partial X \bigr\}$, 
and $D_f$ be the bounded domain in $\hat{X}$ with boundary $\Sigma_f$.
Then $(D_f, \hat{\lambda})$ is a Liouville domain. 

If $| \log f |_{C^0(\partial X)} \le c$,
$X(e^{-c}) \subset D_f \subset X(e^c)$. Hence by theorem \ref{thm:monotonicity},
\[
e^{-c} \le \frac{\capp_S(D_f, \hat{\lambda})}{\capp_S(X,\lambda)} \le e^c.
\]
In particular, if $|\log f|_{C^0(\partial X)}$ is sufficiently small, then 
$\capp_S (D_f,\hat{\lambda})$ is sufficiently close to $\capp_S(X,\lambda)$.

Let $(f_m)_m$ be a sequence of $C^\infty(\partial X)$, such that all periodic Reeb orbits on 
$(\Sigma_{f_m}, \hat{\lambda})$ are non-degenerate, and $|\log f_m|_{C^2(\partial X)} \to 0$ as $m \to \infty$.
By lemma \ref{lem:reebchord}, for each integer $m$ there exists $x_m \in \mca{P}_0(\Sigma_{f_m},\hat{\lambda})$ such that
$\tau(x_m)= \capp_S (D_{f_m},\hat{\lambda})$ and $\mu_\CZ (x_m) \in \{n,n+1\}$.
Since $|\log f_m|_{C^2(\partial X)} \to 0$, $f_m \lambda$ converges to $\lambda$ in $C^2$.
Hence, setting $R_m$ to be the Reeb vector field on $(\partial X, f_m\lambda)$, 
$R_m$ converges to $R$ in $C^1$.
On the other hand, $\tau(x_m)$ converges to $\capp_S(X,\lambda) > 0$.
Hence, up to a subsequence, $(x_m)_m$ converges to $x_\infty \in \mca{P}_0(\partial X, \lambda)$ 
such that 
$\tau(x_\infty)=\capp_S(X,\lambda)$. Moreover, 
\[
\mu_{\CZ}(x_\infty) \le \liminf_{m \to \infty} \mu_{\CZ}(x_m) \le n+1,
\]
where the first inequality follows from 
lemma \ref{lem:i-1}.
\end{proof}

Theorem \ref{thm:reebchord}, together with lemma \ref{lem:CZ} implies the following corollary:

\begin{cor}\label{cor:reebchord}
Let $(X,\lambda)$ be a Liouville domain, and $\capp_S(X,\lambda)<\infty$.
Then, for any $H \in C^\infty(X)$ such that $\partial X = H^{-1}(0)$ and $\partial_r H>0$ on $\partial X$, 
there exists $x: \R/\tau \Z \to \partial X$ such that $\partial_t x=X_H(x)$,
$\int_{\R/\tau \Z} x^*\lambda = \capp_S(X,\lambda)$ and 
$\mu_{\CZ} (x) \le n+1$.
\end{cor}

\subsection{Proof of theorem \ref{thm:covering}}
We prove the first assertion.
Since $\deg \pi <\infty$, $Y$ is compact. 
Define $Z \in \mca{X}(X)$ by $i_Zd\lambda=\lambda$.
Then,
$i_{\pi^*Z}d\pi^*\lambda=\pi^*\lambda$, and 
$\pi^*Z$ points outwards on $\partial Y$. 
Hence $(Y,\pi^*\lambda)$ is a Liouville domain.
Now we prove the second assertion: $\capp_S(Y,\pi^*\lambda) \le \capp_S(X,\lambda)$.
Define $\hat{\pi}: \hat{Y} \to \hat{X}$ by 
\[
\hat{\pi}(y) = \begin{cases}
               \pi(y) &( y \in Y) \\
               \bigl(\pi(z),r\bigr) &\bigl(y=(z,r) \in \partial Y \times [1,\infty)\bigr).
               \end{cases}
\]
Then, $\hat{\pi}: \hat{Y} \to \hat{X}$ is a covering map and $\deg \hat{\pi} = \deg \pi$.

For $H \in \mca{H}(X,\lambda)$, denote $H \circ \hat{\pi}$ by $\bar{H}$.
Since $\mca{P}(H)$ and $\mca{P}(\bar{H})$ consist of contractible solutions, 
$\mca{P}(\bar{H}) \to \mca{P}(H): y \mapsto \hat{\pi} \circ y$ is $\deg(\pi):1$.
We denote this map also by $\hat{\pi}$.

Denote the Poincar\'{e} map generated by $X_H$ (resp. $X_{\bar{H}}$) by $\bigl(\Phi^H_t\bigr)_t$
(resp. $\bigl(\Phi^{\bar{H}}_t\bigr)_t$).
Clearly, 
$d\hat{\pi} \circ d\Phi^{\bar{H}}_1 = d\Phi^H_1 \circ d\hat{\pi}$.
Hence $y \in \mca{P}(\bar{H})$ is non-degenerate if and only if $\hat{\pi}(y) \in \mca{P}(H)$ is non-degenerate.
Moreover, since $\tau(\partial Y,\pi^*\lambda) \subset \tau(\partial X,\lambda)$ and 
$a_{\bar{H}}=a_H$, 
if $a_H \notin \tau(\partial X,\lambda)$ then 
$a_{\bar{H}} \notin \tau(\partial Y,\pi^*\lambda)$.
Therefore, if $H \in \mca{H}_\ad(X,\lambda)$ then
$\bar{H} \in \mca{H}_\ad(Y,\pi^*\lambda)$.

Let $H \in \mca{H}_\ad(X,\lambda)$ and 
$J=(J_t)_{t \in \R/\Z} \in \mca{J}(X,\lambda)$.
Recall that $J$ is said to satisfy the \textit{transversality condition} with respect to $H$ if and only if 
for any $x,x' \in \mca{P}(H)$ and $u \in \hat{\mca{M}}\bigl(x,x':H,J\bigr)$,
\[
D_u: L^{1,p}(u^*T\hat{X}) \to L^p(u^*T\hat{X}); \qquad
\xi \mapsto \nabla_s\xi - J_t\nabla_t\xi_t -\bigl(\nabla_\xi J_t \cdot \partial_t u + \nabla_\xi(\nabla H_t) \bigr) 
\]
is onto ($p$ is an arbitrary real number satisfying $p>2$).
Let $\mca{J}_H(X,\lambda)$ be the set of elements of $\mca{J}(X,\lambda)$ which satisfy 
the transversality condition with respect to $H$.

Define $\bar{J} \in \mca{J}(Y,\pi^*\lambda)$ by 
$\bar{J}_t:=\hat{\pi}^*J_t$. 
Then, for any $x \in \mca{P}(H)$ and $y \in \mca{P}(\bar{H})$, the following map is bijective:
\[
\bigsqcup_{y' \in \hat{\pi}^{-1}(x)} \hat{\mca{M}}\bigl(y',y:\bar{H},\bar{J}\,\,\bigr) \to \hat{\mca{M}}\bigl(x,\hat{\pi}(y):H,J\bigr);
\qquad u \mapsto \hat{\pi} \circ u.
\]
Clearly, $d\hat{\pi} \circ D_u = D_{\pi \circ u} \circ d\hat{\pi}$ for any $u$.
Hence, $J \in \mca{J}_H(X,\lambda)$ if and only if $\bar{J} \in \mca{J}_{\bar{H}}(Y,\pi^*\lambda)$.

Let $H \in \mca{H}_\ad(X,\lambda)$ and $J \in \mca{J}_H(X,\lambda)$. Then, we claim that 
\[
\psi_H: \bigl(C_*(H), \partial_{H,J}\bigr) \to \bigl(C_*(\bar{H}), \partial_{\bar{H}, \bar{J}}\bigr); \qquad
[x] \mapsto \sum_{y \in \hat{\pi}^{-1}(x)} [y]
\]
is a chain map. Let $k$ be an integer and $x \in \mca{P}_k(H)$. Then, 
by definition 
\begin{align*}
\psi_H\bigl(\partial_{H,J}[x]\bigr)&=\psi_H\Biggl(\sum_{x'\in\mca{P}_{k-1}(H)} \sharp \mca{M}\bigl(x,x':H,J\bigr)\cdot[x']\Biggr)\\
&=\sum_{y' \in \mca{P}_{k-1}(\bar{H})} \sharp \mca{M}\bigl(x,\hat{\pi}(y'):H,J\bigr)\cdot[y'],\\
\partial_{\bar{H},\bar{J}}\bigl(\psi_H[x]\bigr)&=\sum_{y \in \hat{\pi}^{-1}(x)} \partial_{\bar{H},\bar{J}}[y]
=\sum_{y' \in \mca{P}_{k-1}(\bar{H})}
\Biggl(\sum_{y \in \hat{\pi}^{-1}(x)} \sharp \mca{M}\bigl(y,y':\bar{H},\bar{J}\,\,\bigr)\Biggr) \cdot [y'].
\end{align*}
Hence it is enough to prove that 
\[
\sharp \mca{M}\bigl(x,\hat{\pi}(y'):H,J\bigr)=\sum_{y \in \hat{\pi}^{-1}(x)} \sharp \mca{M}\bigl(y,y':\bar{H},\bar{J}\,\,\bigr)
\]
for any $x \in \mca{P}_k(H)$, $y' \in \mca{P}_{k-1}(\bar{H})$.
But it is clear since 
\[
\bigsqcup_{y \in \hat{\pi}^{-1}(x)} \mca{M}\bigl(y,y':\bar{H},\bar{J}\,\,\bigr) \to \mca{M}\bigl(x,\hat{\pi}(y'):H,J\bigr); \quad
[u] \mapsto [\hat{\pi} \circ u]
\]
is a bijection. Therefore we have proved that $\psi_H$ is a chain map. 
Hence we can define a morphism
\[
\HF_*(H:X,\lambda) \to \HF_*(\bar{H}:Y,\pi^*\lambda).
\]
We denote this morphism also by $\psi_H$.
Let $H, H' \in \mca{H}_{\ad}(X,\lambda)$. If $a_H \le a_{H'}$, 
\begin{equation}\label{com:covering}
\xymatrix{
\HF_*(H)\ar[rr]^{\psi_H}\ar[d]&&\HF_*\bigl(\bar{H}\bigr)\ar[d]\\
\HF_*(H')\ar[rr]_{\psi_{H'}}&&\HF_*\bigl(\bar{H'}\bigr)
}
\end{equation}
commutes, where vertical morphisms are monotone morphisms.

To prove $\capp_S(Y,\pi^*\lambda) \le \capp_S(X,\lambda)$, it is enough to show that 
if $a \notin \tau(\partial X, \lambda)$ satisfies 
$a> \capp_S(X,\lambda)$, then $a>\capp_S(Y,\pi^*\lambda)$.
$a>\capp_S(X,\lambda)$ implies that 
$\SH_n^{<\delta}(X,\lambda) \to \SH_n^{<a}(X,\lambda)$ vanishes for any $0 < \delta < \delta(\partial X, \lambda)$.
Take $H_{\pm} \in \mca{H}_\ad(X,\lambda)$ such that $a_{H_-}=\delta$, $a_{H_+}=a$.
Then, by lemma \ref{lem:firstdef}, 
$\HF_n(H_-) \to \HF_n(H_+)$ vanishes.

In the rest of this proof, we assume that $X$ and $Y$ are connected
(general case follows at once from this case).
Take $H_-$ so that it satisfies following conditions:
\begin{itemize}
\item $H_-$ is time independent. 
\item $H_-(z,r)=\delta r+\const$ for $(z,r) \in \partial X \times [1,\infty)$.
\item $H_-|_X$ is sufficiently small in $C^2$.
\end{itemize}
Then, $\mca{P}(H_-)$ and $\mca{P}(\bar{H_-})$ consist only of constant maps to $\Crit(H_-)$, $\Crit(\bar{H_-})$.
In particular, $C_k(\bar{H_-})=C_k(H_-)=0$ for $k \ge n+1$.
Hence $\psi_{H_-}: \HF_n(H_-) \to \HF_n(\bar{H_-})$ is injective, therefore isomorphism
(since $X$ and $Y$ are connected).

Hence, the commutative diagram (\ref{com:covering}) implies that 
$\HF_n(\bar{H_-}) \to \HF_n(\bar{H_+})$ vanishes.
Again by lemma \ref{lem:firstdef},
$\SH_n^{<\delta}(Y,\pi^*\lambda) \to \SH_n^{<a}(Y,\pi^*\lambda)$ vanishes.
Hence $a>\capp_S(Y,\pi^*\lambda)$.
This completes the proof of theorem \ref{thm:covering}. \qed

\section{Capacity of Riemannian manifolds}\label{sec:capr}
In this section, we introduce the notion of capacity for Riemannian manifolds without boundaries, 
which is denoted by $\capp_R$.
The main result in this section is theorem \ref{thm:mainestimate}, which includes property (A) which we have 
stated in the introduction.
In 4.1, we give the definition of $\capp_R$, and prove its basic properties.
In particular, proposition \ref{prop:covering}, which is an easy consequence of theorem \ref{thm:covering}, is 
important.
In 4.2, we prove that when $N$ is a compact connected Riemannian manifold with non-empty boundary, then $\capp_R(\interior N)<\infty$
(theorem \ref{thm:w<infty}).
In 4.3, first we prove that $\R^n \setminus \Z^n$ with the flat metric has a finite capacity (theorem \ref{thm:Rn-Zn}).
This is proved by combining proposition \ref{prop:covering} and theorem \ref{thm:w<infty}.
Theorem \ref{thm:mainestimate} is obtained by theorem \ref{thm:Rn-Zn} and elementary geometric arguments.

\subsection{The definition and basic properties}

First we introduce some notations.
Let $N$ be a $n$-dimensional Riemannian manifold.
Let us denote the natural projection $T^*N \to N; (q,p) \mapsto q$ by $\pi_N$.
We define $\lambda_N \in \Omega^1(T^*N)$ by 
\[
\lambda_N(v):=p\bigl(d\pi_N(v)\bigr)\quad\bigl(q \in N, p \in T_q^*N, v \in T_{(q,p)}(T^*N)\bigr).
\]
Then, $\omega_N:=d\lambda_N$ is a symplectic form on $T^*N$.
Define $\nu_N \in \mca{X}(T^*N)$ by $i_{\nu_N}\omega_N=\lambda_N$.

For $V \in C^\infty(N)$,
define $H_V \in C^\infty(T^*N)$ by 
\[
H_V(q,p)=V(q)+|p|^2/2,
\]
and denote $\{H_V \le 0\} \subset T^*N$ by $D_V$.

$\mca{V}(N)$ denotes the set of $V \in C^\infty(N)$ such that 
$0$ is a regular value of $V$, and $\{ V \le 0\} \subset N$ is compact.

For $\xi \in \mca{X}(N)$, define $F_\xi \in C^\infty(T^*N)$ and $\tilde{\xi} \in \mca{X}(T^*N)$ by 
$F_\xi(q,p):=p(\xi_q)$ and $\tilde{\xi}:=X_{F_\xi}$.
Then, $L_{\tilde{\xi}}\omega_N=0$ and $\tilde{\xi}_{(q,0)}=\xi_q$.

\begin{lem}\label{lem:DV}
Let $N$ be a Riemannian manifold without boundary.
Then, for any $V \in \mca{V}(N)$, there exists $\lambda \in \Omega^1(T^*N)$ such that
$d\lambda=\omega_N$ and $(D_V,\lambda)$ is a Liouville domain.
If $V, V' \in \mca{V}(N)$ satisfy $V>V'$, then $\capp_S(D_V,\omega_N) \le \capp_S(D_{V'},\omega_N)$.
\end{lem}
\begin{rem}
For any Liouville domain $(X,\lambda)$, $\capp(X,\lambda)$ depends only on $d\lambda$ (lemma \ref{lem:capacity}, (3)).
Hence formulas 
$\capp_S(D_V,\omega_N)$ and $\capp_S(D_{V'},\omega_N)$ make sense.
\end{rem}
\begin{proof}
We prove the first assertion. 
Take $\xi \in \mca{X}(N)$ such that $dV(\xi)>0$ on $\{V=0\}$.
For $\delta>0$, define $Z_\delta \in \mca{X}(T^*N)$ by
$Z_\delta:=\nu_N + \delta \tilde{\xi}$.
Then, $L_{Z_\delta}\omega_N=\omega_N$ for any $\delta$.
Moreover, $dH_V(Z_\delta)>0$ on $H_V^{-1}(0)$ for sufficiently small $\delta>0$.
Hence $\lambda_\delta:=i_{Z_\delta}\omega_N$ satisfies $d\lambda_\delta=\omega_N$ and $(D_V,\lambda_\delta)$ 
is a Liouville domain for sufficiently small $\delta>0$.

We prove the second assertion. 
If $V>V'$, $\{V=0\} \cap \{V'=0\} = \emptyset$. Hence 
there exists $\xi \in \mca{X}(N)$ such that 
$dV(\xi)>0$ on $\{V=0\}$ and $dV'(\xi)>0$ on $\{V'=0\}$.
Then, for sufficiently small $\delta>0$,
$(D_V,\lambda_\delta)$ and $(D_{V'},\lambda_\delta)$ are both Liouville domains.
On the other hand, $D_V \subset D_{V'}$. Hence 
by theorem \ref{thm:monotonicity}, $\capp_S(D_V,\omega_N) \le \capp_S(D_{V'},\omega_N)$.
\end{proof}

We define the notion of capacity for Riemannian manifolds without boundary.

\begin{defn}\label{defn:width}
Let $N$ be a Riemannian manifold without boundary. 
Then, capacity of $N$ is defined by 
\[
\capp_R(N): = \sup \bigl\{ \capp_S(D_V,\omega_N) \bigm{|} V \in \mca{V}(N), V > -1/2\bigr\}.
\]
\end{defn}

\begin{rem}
As is clear from the above definition, when $N$ is a compact Riemannian manifold without boundary 
$\capp_R(N) = \capp_S(DT^*N,\omega_N)$, where $DT^*N:=\{(q,p) \in T^*N \mid |p| \le 1 \}$.
\end{rem}

In the following, we sometimes denote $N$ by $(N,g)$, where $g$ is the Riemannian metric on $N$.
We also sometimes denote $H_V$ and $D_V$ by $H_{V,g}$, $D_{V,g}$.

\begin{lem}\label{lem:width}
Let $(N,g)$ be a Riemannian manifold without boundary.
\begin{enumerate}
\item[(1)] For any open set $\Omega \subset N$, $\capp_R(\Omega,g) \le \capp_R(N,g)$.
\item[(2)] $\capp_R(N,g)=\sup \bigl\{ \capp_R(\Omega,g) \bigm{|} \text{$\Omega \subset N$ is a open set such that $\bar{\Omega}$ is compact} \bigr\}$.
\item[(3)] Let $a$ be a positive number. Then $\capp_R(N,ag)=a \cdot \capp_R(N,g)$, where $ag$ is defined by $(ag)(v):=a \cdot g(v)\,(v \in TN)$.
\item[(4)] Let $g'$ be a Riemannian metric on $N$, and assume that $g \le g'$ (which means that $g(v) \le g'(v)$ for any 
$v \in TN$). Then $\capp_R(N,g) \le \capp_R(N,g')$.
\end{enumerate}
\end{lem}
\begin{proof}
(1) and (2) are clear from the definition. 
(3) follows from $D_{V,ag}=\bigl\{(q,ap) \bigm{|} (q,p) \in D_{V,g} \bigr\}$.
(4) follows from $g \le g' \implies D_{V,g} \subset D_{V,g'}$.
\end{proof}

\begin{prop}\label{prop:covering}
Let $(N,g)$ be a Riemannian manifold without boundary, and 
$\pi: M \to N$ be a covering map such that $\deg \pi <\infty$.
Then, $\capp_R(M,\pi^*g) \le \capp_R(N,g)$.
\end{prop}
\begin{proof}
For any $V \in \mca{V}(M)$ such that $V>-1/2$, 
there exists $W \in \mca{V}(N)$ such that $W>-1/2$ and 
$V>W \circ \pi$. Hence 
\[
\capp_S(D_V,\omega_M) \le \capp_S (D_{W \circ \pi}, \omega_M) \le \capp_S (D_W,\omega_N) \le \capp_R(N,g).
\]
The first inequality follows from lemma \ref{lem:DV}, the second inequality follows from 
theorem \ref{thm:covering}, and the last inequality is clear from the definition of $\capp_R$.
Therefore $\capp_R(M,\pi^*g)  \le \capp_R(N,g)$.
\end{proof}

\subsection{Capacity of interiors of compact Riemannian manifolds with boundaries}

The goal of this subsection is to prove the following theorem:

\begin{thm}\label{thm:w<infty}
Let $N$ be a compact connected Riemannian manifold with non-empty boundary. Then, $\capp_R(\interior N)<\infty$.
\end{thm}

At first, notice the following consequence of theorem \ref{thm:isotopyinvariance}:

\begin{lem}\label{lem:invariance}
Let $N$ be a Riemannian manifold without boundary, and $V \in \mca{V}(N)$.
Then, $\SH_*(D_V, \omega_N)$ depends only on diffeomorphism type of $\{V \le 0\}$.
\end{lem}

We prove the following lemma:

\begin{lem}\label{lem:SH=0}
Let $N$ be a Riemannian manifold without boundary, and $V \in \mca{V}(N)$.
Assume that $\{V \le 0\}$ is connected and $\{V=0\} \ne \emptyset$.
Then, $\SH_*(D_V,\omega_N)=0$. In particular, $\capp_S (D_V,\omega_N)<\infty$.
\end{lem}
\begin{proof}
First note that the second assertion follows from the first assertion by lemma \ref{lem:capacity}-(2).
We prove the first assertion.
By lemma \ref{lem:invariance}, for any $W \in \mca{V}(N)$ such that 
$\{W \le 0\}=\{V \le 0\}$, $\SH_*(D_V,\omega_N) \cong \SH_*(D_W,\omega_N)$.
Since $\{V \le 0\}$ is a compact connected manifold with non-empty boundary, we can 
take $W$ so that it is a Morse function and 
$\{P_1,\ldots,P_m\}:=\Crit(W) \cap \{W \le 0\}$ satisfies the following:
\begin{itemize}
\item $\min W= W(P_1)<W(P_2)<\cdots<W(P_m)<0$.
\item $\ind P_j \le n-1$ for all $1 \le j \le m$.
\item $\ind P_j =0$ if and only if $j=1$.
\end{itemize}

To complete the proof, we extend the definition of symplectic homology.
Let $(X,\lambda)$ be a Liouville domain and $H$ be a Hamiltonian on its completion. 
Then, let $\HF_*^{\all}(H)$ be the homology of $\bigl(C_*^\all(H),\partial\bigr)$, 
where $C_*^\all(H)$ is a $\Z_2$-graded free $\Z_2$ module generated by all (not only contractible) periodic 
orbits of $X_H$. We define $\SH_*^\all(X,\lambda):= \varinjlim_{H} \HF_*^\all(H)$.
Obviously, $\SH_*^\all(X,\lambda)=0 \implies \SH_*(X,\lambda)=0$.

In the following, we prove that $\SH_*^\all(D_W,\omega_N)=0$.
For $a \in \R$, abbreviate $\{H_W \le a\} \subset T^*N$ by $D_{\le a}$. 
Then, for any $a \in \bigl(W(P_1), W(P_2)\bigr)$, 
\[
\SH_*(D_{\le a},\omega_N) \cong \SH_*\bigl(B^{2n}(1),\omega_n\bigr)=0.
\]
The first isomorphism follows from lemma \ref{lem:invariance}, and the second equality follows from 
theorem \ref{thm:ball}.
Since $D_{\le a}$ is simply connected, $\SH_*^\all(D_{\le a},\omega_N)=0$.

Hence it is enough to show that for any $j \in \{2,\ldots,m\}$, 
$a \in \bigl(W(P_{j-1}), W(P_j)\bigr)$ and 
$b \in \bigl(W(P_j), W(P_{j+1})\bigr)$, the isomorphism 
$\SH_*^\all(D_{\le a},\omega_N) \cong \SH_*^\all(D_{\le b},\omega_N)$ holds.
By lemma \ref{lem:invariance}, it is enough to show that there exists $\varepsilon>0$ such that
$\SH_*^\all(D_{\le W(P_j)-\varepsilon},\omega_N) \cong \SH_*^\all(D_{\le W(P_j)+\varepsilon},\omega_N)$.

Let $j \in \{2,\ldots,m\}$ and set $k:=\ind P_j$.
Take a local coordinate $(q_1,\ldots,q_n)$ around $P_j$ such that $P_j$ corresponds to $(0,\ldots,0)$ and
\[
W(q_1,\ldots,q_n)= W(P_j)-\sum_{1 \le i \le k} q_i^2 + \sum_{k+1 \le i \le n} q_i^2.
\]
Take $\varepsilon>0$ sufficiently small, and let
\[
\Sigma_\varepsilon:=\bigl\{(q_1,\ldots,q_k,0,\ldots,0) \in N \bigm{|} q_1^2+ \cdots +q_k^2 = \varepsilon \bigr\}.
\]
For $\delta>0$, let $Z_\delta:=\nu_N + \delta \widetilde{\nabla{W}}$, and
$\lambda_\delta:=i_{Z_\delta}\omega_N$.
Then, $(D_{\le W(P_j)\pm\varepsilon},\lambda_\delta)$ are 
Liouville domains for sufficiently small $\delta>0$, and
$\Sigma_\varepsilon$ is an isotropic submanifold of $(\partial D_{\le W(P_j)-\varepsilon}, \lambda_\delta)$.
Moreover, $(D_{\le W(P_j)+\varepsilon},\lambda_\delta)$ is isotopic as Liouville domain (see definition \ref{defn:isotopic}) to the
Liouville domain obtained by attaching $k$-handle to $(D_{\le W(P_j)-\varepsilon},\lambda_\delta)$ along
$\Sigma_\varepsilon$ in the sense of \cite{Weinstein}.
Hence by theorem 1.11 (1) in \cite{Cieliebak},
\[
\SH_*^\all(D_{\le W(P_j)-\varepsilon},\omega_N) \cong \SH_*^\all(D_{\le W(P_j)+\varepsilon},\omega_N).
\]
This completes the proof.
\end{proof}
\begin{rem}
The above proof shows that $(D_W,\omega_N)$ carries a structure of so called "subcritical Weinstein domain".
\end{rem}

Finally, we prove theorem \ref{thm:w<infty}.

\begin{proof}
For any Riemannian metrics $g$ and $g'$ on $N$, there exists $a>0$ such that $g \le ag'$ since $N$ is compact.
Then, $\capp_R(\interior N,g) \le a\cdot\capp_R(\interior N, g')$.
Therefore it is enough to show that there exists a Riemannian metric $g$ on $N$ such that 
$\capp_R(\interior N, g)<\infty$.

Take a Riemannian manifold $(N',g')$ without boundary, and 
an embedding $i: N \hookrightarrow N'$. 
We show that $\capp_R(\interior N, i^*g')<\infty$.

Since $N$ is a compact connected manifold with non-empty boundary, 
there exists $V \in \mca{V}(N')$ such that $V \circ i<-1/2$ and 
$\{V \le 0\}$ is connected, $\{V=0\} \ne \emptyset$.
For any $W \in \mca{V}(\interior N)$ such that $W>-1/2$, 
$\capp_S(D_W,\omega_N) \le \capp_S (D_V,\omega_{N'})$.
Hence $\capp_R(\interior N) \le \capp_S (D_V,\omega_{N'})$.
On the other hand, $\capp_S (D_V,\omega_{N'})<\infty$ by lemma \ref{lem:SH=0}.
This completes the proof.
\end{proof}

The following corollary of theorem \ref{thm:w<infty} plays an important role in the next subsection.

\begin{cor}\label{cor:w<infty}
Let $N$ be a compact connected Riemannian manifold (possibly with boundary), and $x \in \interior N$.
Then, $\capp_R\bigl(\interior N \setminus \{x\}\bigr)<\infty$.
\end{cor}
\begin{proof}
Let $n:=\dim N$.
When $n=1$, the assertion is easily confirmed.
Hence in the following, we consider the case $n \ge 2$.
It is enough to show that there exists a Riemannian metric $g$ on $N$ such that
$\capp_R\bigl(\interior N \setminus \{x\},g\bigr)<\infty$.

Let $U$ be a coordinate neighborhood containing $x$, and $(q_1,\ldots,q_n)$ be a local coodinate on $U$, 
such that $x$ corresponds to $(0,\ldots,0)$.
We may assume that $B(x,1) \subset U$ and 
$dg=\sum_{1 \le j \le n} dq_j^2$ on $U$.

Set $S:=\partial B(x,1) \subset U$, and $g_S:=g|_S$.
Take arbitrary smooth function $\mu:[0,1] \to \R_{>0}$ such that 
$\mu \equiv 1/2$ on $[0,1/3]$ and 
$\mu(r) = r$ on $[2/3, 1]$, and $\mu(r) \ge r$.
Consider a cylinder 
$C=S \times [0,1]$ equipped with a metric $h$ defined by 
\[
\big\lvert v+a\partial_r(z,r) \big\rvert_h
:=\bigl( |v(z)|_{g_S}^2 \mu(r)^2 + a^2 \bigr)^{1/2} \qquad(v \in TS, a \in \R).
\]
Then, 
\[
I: \bigl( S \times (2/3, 1] ,h \bigr) \to \bigl(B(x,1) \setminus B(x,2/3),g \bigr);\quad (z,r) \mapsto zr
\]
is an isometry. 

Let $(\tilde{N},\tilde{g})$ be a Riemannian manifold which is obtained by pasting $\bigl(N \setminus B(x,2/3), g\bigr)$ 
with $(C,h)$ via $I$.
Then, $\tilde{N}$ is a compact manifold with non-empty boundary (since $S \times \{0\} \subset \partial \tilde{N}$), 
and connected (since $N$ is connected and $n \ge 2$).
Hence theorem \ref{thm:w<infty} implies that 
$\capp_R(\interior \tilde{N},\tilde{g})<\infty$.

Define a diffeomorphism $J : \interior \tilde{N} \to \interior N \setminus \{x\}$ by
\[
J(y) := \begin{cases}
        y &\bigl( y \in N \setminus B(x,2/3) \bigr) \\ \\
       rz &\bigl( y=(z,r) \in C=S \times (0,1] \bigr)
       \end{cases}.
\]
Then, since $\mu(r) \ge r$, $J^*g \le \tilde{g}$. Hence
\[
\capp_R\bigl(\interior N \setminus \{x\}, g\bigr)
=\capp_R(\interior \tilde{N}, J^*g )
\le \capp_R(\interior \tilde{N}, \tilde{g}) <\infty.
\]
\end{proof}

\subsection{Capacity of domains in $\R^n$}

\begin{thm}\label{thm:Rn-Zn}
Let $g_n$ denote the flat metric on $\R^n$. Then,
$\capp_R \bigl(\R^n \setminus \Z^n, g_n) <\infty$.
\end{thm}
\begin{proof}
When $n=1$, $\capp_R(\R \setminus \Z, g_1)=\capp_R\bigl((0,1), g_1\bigr)<\infty$.
In the following, we assume that $n \ge 2$.

By lemma \ref{lem:width}-(2), 
it is enough to show that, for any bounded open set $\Omega$ in ${\R}^n \setminus {\Z}^n$, 
$\capp_R(\Omega)$ is bounded from above by some constant which depends only on $n$.

Let $\Omega$ be a bounded open set in $\R^n \setminus \Z^n$.
Then, for sufficiently large integer $l$, $\Omega \subset (-l,l)^n$.
Hence $\Omega$ can be considered as an open set in 
$(\R^n \setminus \Z^n)/2l{\Z}^n$.

Consider the natural covering map of degree $(2l)^n$:
$(\R^n \setminus \Z^n)/2l{\Z}^n \to (\R^n \setminus \Z^n)/{\Z}^n$.
Then 
\[
\capp_R(\Omega) \le \capp_R\bigl((\R^n \setminus \Z^n)/2l{\Z}^n\bigr)
\le \capp_R\bigl((\R^n \setminus \Z^n)/{\Z}^n \bigr).
\]
The first inequality follows from lemma \ref{lem:width}-(1), and the second inequality follows from proposition \ref{prop:covering}.
On the other hand, $(\R^n \setminus \Z^n)/\Z^n$ is $\R^n/\Z^n$ minus a point.
Hence, by corollary \ref{cor:w<infty}, 
$\capp_R\bigl((\R^n \setminus \Z^n)/\Z^n \bigr)<\infty$.
This completes the proof.
\end{proof}


\begin{thm}\label{thm:mainestimate}
For each integer $n$, there exists $c_0(n), c_1(n)>0$ such that for any non-empty open set $\Omega$ in ${\R}^n$,
\[
c_0(n) \le \frac{\capp_R(\Omega,g_n)}{r(\Omega)} \le c_1(n).
\]
\end{thm}
\begin{proof}
If $r<r(\Omega)$, then there exists $x \in \R^n$ such that $B(x,r) \subset \Omega$. 
Hence, by lemma \ref{lem:width}-(1) and (3), 
\[
\capp_R(\Omega) \ge \capp_R\bigl(B(x,r)\bigr) = r \cdot \capp_R\bigl(B(x,1)\bigr).
\]
Hence $\capp_R\bigl(B^n(1)\bigr) \le \frac{\capp_R(\Omega)}{r(\Omega)}$ for any $\Omega \ne \emptyset$.

Next we bound $\frac{\capp_R(\Omega)}{r(\Omega)}$ from above.
Take an arbitrary positive number $r$ so that $r>r(\Omega)$.
Then, for any $x \in {\R}^n$, $B(x,r) \setminus \Omega \ne \emptyset$.
For any $j=(j_1,\ldots,j_n) \in \Z^n$, take an arbitrary point $p_j$ on 
$B\bigl(4rj,r\bigr) \setminus \Omega$, where $4rj=(4rj_1,\ldots,4rj_n)$.
Then, $\capp_R(\Omega) \le \capp_R\bigl({\R}^n \setminus \{p_j\}_{j \in \Z^n}\bigr)$.

Take sufficiently large $\alpha_n>0$ so that for any $x \in B^n(1)$, there exists a diffeomorphism $\varphi$ on $B^n(2)$ 
with compact support such that $\varphi(x)=(0,\ldots,0)$, and $g_n \le \alpha_n \cdot \varphi^*g_n$.
Then, since $B(4ri,2r) \cap B(4rj,2r)=\emptyset$ when $i \ne j$, 
there exists a diffeomorphism 
$\psi: {\R}^n \setminus \{p_j\}_{j \in \Z^n} \to {\R}^n \setminus 4r{\Z}^n$ such that 
$g_n \le \alpha_n \cdot  \psi^*g_n$.
Then, 
\begin{align*}
&\capp_R\bigl({\R}^n \setminus \{p_j\}_{j \in \Z^n}, g_n\bigr)
\le \alpha_n \cdot \capp_R\bigl({\R}^n \setminus \{p_j\}_{j \in \Z^n}, \psi^*g_n \bigr) \\
&\qquad=\alpha_n \cdot \capp_R(\R^n \setminus 4r{\Z}^n, g_n)
=4\alpha_n r \cdot \capp_R(\R^n \setminus {\Z}^n, g_n).
\end{align*}
To sum up, 
\[
r>r(\Omega) \implies \capp_R(\Omega) \le 4\alpha_n r \cdot \capp_R(\R^n \setminus {\Z}^n).
\]
Hence $\frac{\capp_R(\Omega)}{r(\Omega)} \le 4\alpha_n  \cdot \capp_R(\R^n \setminus {\Z}^n) <\infty$.
\end{proof}

\section{Short periodic billiard trajectory}\label{sec:billiard}

The goal of this section is to prove the following theorem.

\begin{thm}\label{thm:billiard}
Let $\Omega$ be a bounded domain in ${\R}^n$ with smooth boundary.
Then, there exists a periodic billiard trajectory on $\Omega$ with at most $n+1$ bounce times
and length equal to $\capp_R(\Omega)$.
\end{thm}

Theorem \ref{thm:billiard} is exactly the same as property (B) of $\capp_R$ which we have introduced in the introduction.
Hence, as we have explained in the introduction, it completes the proof of our main theorem \ref{mainthm}.

We start to prove theorem \ref{thm:billiard}.
The proof heavily relies on the arguments in \cite{AlbersMazzucchelli}.
First we recall the settings in \cite{AlbersMazzucchelli}.
Fix $d_0 \in (0,1/2)$ so small that
$\dist_{\partial\Omega}: q \mapsto \min\bigl\{|q-q'| \bigm{|} q' \in \partial \Omega \bigr\}$ 
is smooth on $\{\dist_{\partial\Omega} \le 2d_0\}$.
Let $k: [0,\infty) \to [0,2d_0]$ be a smooth function such that $0 \le k' \le 1$, $k(x)=x$ if $x \le d_0$ and
$k$ is constant on $[2d_0,\infty)$. Then, we define a function $h \in C^\infty(\bar{\Omega})$ by 
$h(q):=k\bigl(\dist_{\partial \Omega}(q)\bigr)$, and define $U \in C^\infty(\Omega)$ by 
$U(q):=h^{-2}(q)$. Then, $U$ is a positive function on $\Omega$ which grows like $(\dist_{\partial \Omega})^{-2}$ 
near $\partial \Omega$ and is constant on the region $\bigl\{\dist_{\partial \Omega} \ge 2d_0 \bigr\}$.

For each $\varepsilon>0$, consider the modified Lagrangian
\[
L_\varepsilon: T\Omega \to \R; \, (q,v) \mapsto |v|^2/2 - \varepsilon U(q).
\]
For each energy value $E \in \R$ the free-time action functional $\mca{L}^E_\varepsilon: L^{1,2}(\R/\Z,\Omega) \times \R_{>0}
\to \R$ is given by 
\[
\mca{L}^E_\varepsilon(\Gamma,\tau):= \tau \int_0^1 \biggl[ L_\varepsilon\bigl(\Gamma(t),\tau^{-1}\partial_t\Gamma(t)\bigr)
+E \biggr] dt.
\]
For $(\Gamma,\tau) \in L^{1,2}(\R/\Z,\Omega) \times \R_{>0}$, let $\gamma$ be the corresponding $\tau$-periodic curve, i.e.
$\gamma: \R/\tau \Z \to \Omega;\,t \mapsto \Gamma(t/\tau)$.
Then, straightforward calculations show that $(\Gamma,\tau)$ is a critical point of $\mca{L}^E_\varepsilon$ if and only if 
$\gamma$ satisfies 
\[
\partial_t^2\gamma+ \nabla(\varepsilon U)(\gamma)=0
\]
with energy
\[
E_\varepsilon(\gamma):=|\partial_t\gamma|^2/2 + \varepsilon U(\gamma)=E.
\]
When $(\Gamma,\tau)$ is a critical point of $\mca{L}^E_\varepsilon$, 
$\mu_\Morse(\Gamma:\mca{L}^E_\varepsilon|_{L^{1,2}(\R/\Z) \times \{\tau\}})$ denotes the number of negative eigenvalues 
of the Hessian of $\mca{L}^E_\varepsilon|_{L^{1,2}(\R/\Z) \times \{\tau\}}$ at $(\Gamma,\tau)$.

For $\varepsilon>0$, define $H_\varepsilon : T^*\Omega \to \R$ by 
$H_\varepsilon(q,p):=\varepsilon U(q)+|p|^2/2$, and
$D_\varepsilon:=\{H_\varepsilon \le 1/2 \bigr\} \subset T^*\Omega$.

\begin{lem}\label{lem:index}
For any $\varepsilon>0$, there exists 
$(\Gamma_\varepsilon,\tau_\varepsilon)$, a critical point of $\mca{L}^{1/2}_\varepsilon$,
 which satisfies the following properties:
\begin{enumerate}
\item[(1)] $\int_{\R/\tau_\varepsilon \Z} | \partial_t\gamma_\varepsilon|^2dt = \capp_S(D_\varepsilon,\omega_\Omega)$.
\item[(2)] $\mu_{\Morse}(\Gamma_\varepsilon; \mca{L}^{1/2}_\varepsilon|_{L^{1,2}(\R/\Z) \times \{\tau_\varepsilon\}}) \le n+1$.
\end{enumerate}
\end{lem}
\begin{proof}
Take arbitrary $\lambda \in \Omega^1(T^*\Omega)$ such that $d\lambda=\omega_\Omega$ and
$(D_\varepsilon, \lambda)$ is a Liouville domain.
By corollary \ref{cor:reebchord}, there exists 
$x_\varepsilon:\R/\tau_\varepsilon \Z \to H_\varepsilon^{-1}(1/2)$, 
which is a periodic orbit of $X_{H_\varepsilon}$ and satisfies 
\[
\int_{\R/\tau_\varepsilon \Z} x_\varepsilon^*\lambda= \capp_S (D_\varepsilon, \omega_\Omega), \qquad
\mu_{\CZ}(x_\varepsilon) \le n+1.
\]
Define $\gamma_\varepsilon :\R/\tau_\varepsilon \Z \to \Omega$ and
$\Gamma_\varepsilon: \R/\Z \to \Omega$ by 
$\gamma_\varepsilon:=\pi_\Omega \circ x_\varepsilon$, $\Gamma_\varepsilon(t):=\gamma_\varepsilon(\tau_\varepsilon t)$.
Then, it is obvious that $(\Gamma_\varepsilon,\tau_\varepsilon)$ is a critical point of $\mca{L}_\varepsilon^{1/2}$.
Moreover, 
\[
\int_{\R/\tau_\varepsilon \Z} |\partial_t \gamma_\varepsilon|^2 dt 
=\int_{\R/\tau_\varepsilon \Z} x_\varepsilon^*\lambda_\Omega
=\int_{\R/\tau_\varepsilon \Z} x_\varepsilon^*\lambda
=\capp_S(D_\varepsilon, \omega_\Omega).
\]
In the second equality, we use that $x_\varepsilon$ is contractible in $T^*\Omega$ and
$d\lambda=d\lambda_\Omega$.

Finally,
\[
\mu_{\Morse}(\Gamma_\varepsilon; \mca{L}^{1/2}_\varepsilon|_{L^{1,2}(\R/\Z) \times \{\tau_\varepsilon\}})
=\mu_{\CZ}(x_\varepsilon).
\]
This identity follows from theorem 7.3.1 in \cite{Long}.
(2) follows immdiately from this identity and $\mu_{\CZ} (x_\varepsilon) \le n+1$.
\end{proof}

\begin{lem}\label{lem:period}
For each $\varepsilon>0$, take $(\Gamma_\varepsilon, \tau_\varepsilon)$, a critical point of $\mca{L}^{1/2}_\varepsilon$
which satisfies properties in lemma \ref{lem:index}.
Then, 
\[
0< \liminf_{\varepsilon \to 0} \tau_\varepsilon  \le \limsup_{\varepsilon \to 0} \tau_\varepsilon <\infty.
\]
\end{lem}
\begin{proof}
First we show that $\liminf_{\varepsilon \to 0} \tau_\varepsilon>0$.
Assume that $\liminf_{\varepsilon \to 0} \tau_\varepsilon=0$, i.e.
there exists a sequence $(\varepsilon_k)_k$ such that 
$\varepsilon_k, \tau_{\varepsilon_k} \to 0$ as $k \to \infty$.
Then, there exists a sequence of integers $(N_k)_k$ such that $1<\tau_{\varepsilon_k}N_k<2$.
Set $\Theta_k \in L^{1,2}(\R/\Z,\Omega)$ by
$\Theta_k(t):=\Gamma_{\varepsilon_k}(N_k t)$.
Then, $(\Theta_k,\tau_{\varepsilon_k}N_k)$ is a critical point of $\mca{L}^{1/2}_{\varepsilon_k}$.
By proposition 2.1 in \cite{AlbersMazzucchelli}, a certain subsequence of $(\Theta_k)_k$ converges 
in $L^{1,2}(\R/\Z,\bar{\Omega})$. On the other hand, by lemma \ref{lem:index}-(1), 
\[
|\partial_t\Theta_k|^2_{L^2(\R/\Z)}
=N_k^2\int_{\R/\Z} \big\lvert \partial_t\Gamma_{\varepsilon_k}(t) \big\rvert^2 dt
=N_k^2\tau_{\varepsilon_k} \capp_S(D_{\varepsilon_k},\omega_\Omega).
\]
Since $\capp_S(D_{\varepsilon_k},\omega_\Omega) \to \capp_R(\Omega)$ as $k \to \infty$ and 
$N_k\tau_{\varepsilon_k}>1$, the last term goes to $\infty$ as $k \to \infty$, 
contradicting that a certain subsequence of $(\Theta_k)_k$ converges in $L^{1,2}(\R/\Z,\bar{\Omega})$.

Next we show that $\limsup_{\varepsilon \to 0} \tau_\varepsilon<\infty$.
For each $\varepsilon>0$, define $x_\varepsilon :\R/\tau_\varepsilon \Z \to T^*\Omega$ by 
$x_\varepsilon=(\gamma_\varepsilon, \partial_t\gamma_\varepsilon)$.
Then, $x_\varepsilon$ is an integral curve of $X_{H_\varepsilon}$ on $H_\varepsilon^{-1}(1/2)$.
On the other hand, by proposition 3.2 in \cite{AlbersMazzucchelli}, when $\varepsilon>0$ is sufficiently small,
there exists $\lambda_\varepsilon \in \Omega^1(T^*\Omega)$ such that 
$d\lambda_\varepsilon=\omega_\Omega$ and the following inequality holds on $H_\varepsilon^{-1}(1/2)$:
\[
\lambda_\varepsilon(X_{H_\varepsilon}) \ge \frac{(1/2-0)^3}{2\bigl[(1/2-0)^2+48(1/2-0)^2\bigr]} = \frac{1}{196}.
\]
Notice that $(D_\varepsilon, \lambda_\varepsilon)$ is a Liouville domain, since setting $Z_\varepsilon \in \mca{X}(T^*\Omega)$
by $i_{Z_\varepsilon}\omega_\Omega=\lambda_\varepsilon$, then
$dH_{\varepsilon}(Z_\varepsilon)=\omega_\Omega(Z_\varepsilon, X_{H_\varepsilon})=\lambda_\varepsilon(X_{H_\varepsilon})>0$
on $H_\varepsilon^{-1}(1/2)$.

Since
\[
\int_{\R/\tau_\varepsilon \Z} \lambda_\varepsilon\bigl(X_{H_\varepsilon}(x_\varepsilon(t))\bigr) dt =
\int_{\R/\tau_\varepsilon \Z} x_\varepsilon^* \lambda_\varepsilon 
= \capp_S(D_\varepsilon,\omega_\Omega),
\]
$\tau_\varepsilon \le 196 \cdot \capp_S(D_\varepsilon, \omega_\Omega) \le 196 \cdot \capp_R(\Omega)$ for sufficiently 
small $\varepsilon>0$. This completes the proof.
\end{proof}

Finally, we prove theorem \ref{thm:billiard}.

\begin{proof}
For each $\varepsilon>0$, take $(\Gamma_\varepsilon, \tau_\varepsilon)$, a critical point of $\mca{L}^{1/2}_{\varepsilon}$
which satisfies properties in lemma \ref{lem:index}. 
Then by lemma \ref{lem:period}, we can apply proposition 2.1 in \cite{AlbersMazzucchelli} to the sequence 
$(\Gamma_\varepsilon, \tau_\varepsilon)_{\varepsilon>0}$.
i.e. a certain subsequence of $(\Gamma_\varepsilon, \tau_\varepsilon)_{\varepsilon>0}$ converges to 
$(\Gamma, \tau)$ in $L^{1,2}(\R/\Z, \bar{\Omega}) \times \R_{>0}$, and $\gamma:\R/\tau \Z \to \bar{\Omega}; t \mapsto 
\Gamma(t/\tau)$ is a periodic billiard trajectory on $\Omega$,
such that $E(\gamma)=\lim_{\varepsilon \to 0} E_\varepsilon(\gamma_\varepsilon) = 1/2$.

Since $\mu_{\Morse}(\Gamma_\varepsilon; \mca{L}^{1/2}_\varepsilon|_{L^{1,2}(\R/\Z) \times \{\tau_\varepsilon\}}) \le n+1$ for
each $\varepsilon>0$, proposition 2.2 in \cite{AlbersMazzucchelli} implies that $\Gamma$ has at most $n+1$ 
bounce times.

Finally we prove that $\tau=\capp_R(\Omega)$. For each $\varepsilon>0$, by lemma \ref{lem:index}-(1), 
\[
\capp_S(D_\varepsilon, \omega_\Omega)
=\int_{\R/\tau_\varepsilon \Z}|\partial_t\gamma_\varepsilon|^2 dt
=\tau_\varepsilon^{-1} \int_{\R/\Z} |\partial_t\Gamma_\varepsilon|^2 dt.
\]
Since $\lim_{\varepsilon \to 0} \tau_\varepsilon=\tau$ and 
$\lim_{\varepsilon \to 0} \Gamma_\varepsilon \to \Gamma$ in $L^{1,2}(\R/\Z,\bar{\Omega})$, by taking limit of the
above identity we get 
\[
\capp_R(\Omega) =\lim_{\varepsilon \to 0} \capp_S(D_\varepsilon,\omega_\Omega)= \tau^{-1} \int_{\R/\Z} |\partial_t \Gamma|^2 dt.
\]
On the other hand, since $E(\gamma)=1/2$, $|\partial_t \gamma|=1$ for almost every $t \in \R/\tau\Z$.
Hence $|\partial_t \Gamma|=\tau$ for almost every $t \in \R/\Z$.
Therefore
$\capp_R(\Omega)=\tau^{-1} \cdot \tau^2 = \tau$.
\end{proof}

\textbf{Acknowledgements.}
The author would like to appreciate 
his advisor professor Kenji Fukaya for reading the manuscript and 
giving precious comments, 
and also 
the anonymous referee for many suggestions.
The author is supported by Grant-in-Aid for JSPS fellows.

\section*{Appendix: proof of theorem \ref{thm:truncated}}\label{sec:appendix}
Assume that $(X,\lambda)$, $(X,\lambda')$ are Liouville domains such that $d\lambda=d\lambda'$.

Take arbitrary smooth function $\rho \colon \R \to [0,1]$ such that for sufficiently large $s_0>0$
$\rho(s) = \begin{cases}
            0 &(s \le -s_0) \\
            1 &(s \ge s_0) 
           \end{cases}$.
Let $\lambda_s:=\bigl(1-\rho(s)\bigr)\lambda+\rho(s)\lambda'$. 
Then, $(X,\lambda_s)_{s \in \R}$ is a smooth family of Liouville domains.
Denote the completion of $(X,\lambda_s)$ by $(\hat{X}_s, \hat{\lambda}_s)$.

Our aim is to define a morphism from the
Floer chain complex on $(X,\lambda)$ to the Floer chain complex on $(X,\lambda')$.
To define a morphism, we study the Floer equation on a fiber bundle over $\R$, which is constructed as 
follows. 
First, consider trivial bundles 
\[
E_X: X \times \R \to \R , \qquad 
E_{\partial X}: \bigl(\partial X \times (0,\infty)\bigr) \times \R \to \R.
\]
Define an embedding 
\[
I: \bigl(\partial X \times (0,1]\bigr) \times \R \to X \times \R; \quad 
\bigl((z,r),s\bigr) \mapsto \bigl(I_s(z,r), s \bigr)
\]
by ($Z_s$ denotes the vector field on $X$ characterized by $i_{Z_s}d\lambda_s = \lambda_s$):
\begin{align*}
I_s(z,1)&=z  \qquad\qquad\qquad\,\,\, (z \in \partial X), \\
\partial_rI_s(z,r)&=r^{-1}Z_s\bigl(I_s(z,r)\bigr) \qquad \bigl(z \in \partial X, r \in (0,1]\bigr).
\end{align*}

Let $E:=E_X \cup_I E_{\partial X}$. 
$E$ is a fiber bundle over $\R$, and each 
fiber $E_s$ is identified with $\hat{X}_s$. 
Note that there exist natural bundle maps over $\R$:
\begin{align*}
j_1&: E_X \to E, \\
j_2&: \bigl(\partial X \times [1,\infty) \bigr) \times \R \to E.
\end{align*}

To study the Floer equation on $E$, we equip $E$ with a connection $\nabla$, and denote the horizontal lift of 
$\partial_s$ to $E$ by $W$.
We take $\nabla$ so that $W$ satisfies 
\begin{itemize}
\item $j_1^*(W)=(0,\partial_s)$, 
\item $j_2^*(W)=(0,\partial_s)$ outside $\bigl(\partial X \times [1,2]\bigr) \times [-s_0, s_0]$.
\end{itemize}

Let $H_- \in \mca{H}^\rest_\ad(X,\lambda)$, $H_+ \in \mca{H}^\rest_\ad(X,\lambda')$ and 
$(H_s)_{s \in \R}$ be a family of Hamiltonians with the following properties:
\begin{itemize}
\item $H_s \in \mca{H}^{\rest}(X,\lambda_s)$ for any $s \in \R$.
\item $\partial_s H_{s,t}(x)\ge 0$ for any $(s,t) \in \R \times \R/\Z$ and $x \in X$.
\item There exists $s_1 \ge s_0$ such that 
$H_s = \begin{cases}
       H_- & (s \le -s_1) \\
       H_+ & (s \ge s_1)
       \end{cases}$.
\item $\partial_s a_{H_s} \ge 0$.
\end{itemize}

Let $(J_s)_s$ be a family of (time-dependent) almost complex structures on $E_s=\hat{X_s}$, such that 
$J_s \in \mca{J}(X,\lambda_s:1)$ for any $s$, and
\[
J_s = \begin{cases}
      J_{-s_0} &(s \le -s_0) \\
      J_{s_0}  &(s \ge s_0)
      \end{cases}.
\]
We denote $J_{\pm s_0}$ by $J_{\pm}$.

Then, for $x_- \in \mca{P}(H_-)$, $x_+ \in \mca{P}(H_+)$, we study the following Floer equation for 
$u: \R \times \R/\Z \to E$:
\begin{itemize}
\item $u(s,t) \in E_s$,
\item $\nabla_s u - J_{s,t}(\partial_t u - X_{H_{s,t}} \circ u) = 0$, 
\item $u(s) \to x_{\pm}$ as $s \to \pm \infty$.
\end{itemize}
We denote the moduli space of solutions of the above Floer equation by $\hat{\mca{M}}\bigl(x,y:(H_s,J_s)_s\bigr)$.

In the following, we abbreviate $a_{H_s}$ by $a(s)$. 
The key step in the proof of theorem \ref{thm:truncated} is to prove the following lemma:

\begin{lem}\label{lem:C0bound-3}
There exist $c_0, c_1>0$, which depend only on $(J_s)_s$ and $\rho$, such that:
if $a$ satisfies $\partial_s a \ge c_0a +c_1$
on $[-s_0, s_0]$, then for any  
$x_- \in \mca{P}(H_-)$, $x_+ \in \mca{P}(H_+)$ and 
$u \in \hat{\mca{M}}\bigl(x_-,x_+:(H_s,J_s)_s\bigr)$, 
$u(\R \times \R/\Z) \subset j_1(X \times \R)$.
\end{lem}
\begin{proof}
\textbf{Step 1}.
First note that for any $x \in \mca{P}(H_-) \cup \mca{P}(H_+)$, 
$x(\R/\Z) \subset X$. This is because $H_\pm \in \mca{H}^\rest_\ad(X,\lambda)$.
Our aim is to show that $u(\R \times \R/\Z)$ is contained in $j_1(X \times \R)$.
If this is not true, for some $r_0>1$
\[
D_{r_0}: = u^{-1}\Bigl( j_2\bigl(\partial X \times [r_0,\infty)\bigr) \times \R \Bigr)
\]
is a non-empty surface with boundary. 
Note that $D_{r_0}$ must be compact, since both $x_-(\R/\Z)$ and $x_+(\R/\Z)$ are contained in 
$X$.
Define $v: D_{r_0} \to \partial X \times [1,\infty)$ by 
\[
u(s,t) = j_2\bigl(v(s,t), s \bigr), 
\]
and define $z: D_{r_0} \to \partial X$ and $r: D_{r_0} \to [1,\infty)$ by 
$v(s,t)=\bigl(z(s,t), r(s,t)\bigr)$.

We will prove that there exist $c_0, c_1>0$, which depend only on $(J_s)_s$ and $\rho$, 
such that if $a$ satisfies $\partial_s a  \ge c_0a+c_1$ on $[-s_0,s_0]$, 
then $\Delta r \ge 0$. Since $r \equiv r_0$ on $\partial D_{r_0}$, this implies 
$r \le r_0$ on $D_{r_0}$. On the other hand, by definition $r>r_0$ on $\interior D_{r_0}$, hence 
we get a contradiction.

\textbf{Step 2}.
We calculate $\Delta r(s,t)$ for $(s,t) \in D_{r_0}$.
Recall that $u$ satisfies the Floer equation 
\[
\nabla_s u - J_{s,t}\bigl(\partial_t u - X_{H_{s,t}}(u)\bigr) =0.
\]
Since $H_{s,t}(z,r)=a(s)r+b(s)$ on $\partial X \times [1,\infty)$, 
\[
J_{s,t}X_{H_{s,t}} = -\nabla_{s,t} H_{s,t} = -ar \partial_r.
\]
$\nabla_{s,t} H_{s,t}$ denotes the gradient of $H_{s,t}$ with respect to $\langle\,,\,\rangle_{J_{s,t}}$.

Moreover, by definition of $W$, 
\[
\partial_s u = \nabla_s u + W(u).
\]
Hence, the Floer equation can be written as:
\begin{equation}\label{eq:Floer}
\partial_s u - W(u) - J_{s,t}\partial_t u - ar\partial_r =0.
\end{equation}
It is convinient to define $\hat{\lambda} \in \Omega^1(E)$ by
\begin{itemize}
\item $\hat{\lambda}|_{E_s} = \hat{\lambda}_s$ for any $s \in \R$.
\item $\hat{\lambda}(W) \equiv 0$.
\end{itemize}

Then, by applying $dr$ and $\hat{\lambda}$ to (\ref{eq:Floer}) respectively, we get 
\begin{align*}
&\partial_s r + \hat{\lambda}(\partial_t u) - ar - dr\bigl(W(u)\bigr) =0, \\
&\hat{\lambda}(\partial_s u) - \partial_t r =0.
\end{align*}
Then, 
\begin{align*}
\Delta r &= \partial_t\bigl(\hat{\lambda}(\partial_s u)\bigr) - \partial_s\bigl(\hat{\lambda}(\partial_t u)\bigr) 
+ \partial_s(ar) + \partial_s\bigl(dr\bigl(W(u)\bigr)\bigr)\\
&=d\hat{\lambda}(\partial_t u, \partial_s u) + \partial_s(ar) + \partial_s\bigl(dr\bigl(W(u)\bigr)\bigr) \\
&=|\nabla_s u|_{J_{s,t}}^2 + d\hat{\lambda}(\partial_t u, W(u)) + \partial_s a \cdot r 
+a \cdot dr(W(u)) + \partial_s \bigl(dr(W(u)) \bigr).
\end{align*}
In the following, we abbreviate $|\cdot|_{J_{s,t}}$ by $|\cdot|_{s,t}$.

\textbf{Step 3.} We prove that $\Delta r(s,t) \ge 0$ when $s \notin [-s_0, s_0]$.
Assume that $s \notin [-s_0, s_0]$. 
Then, since $j_2^*W=(0,\partial_s)$ on $\bigl(\partial X \times [1,\infty)\bigr) \times \bigl(\R \setminus [-s_0,s_0]\bigr)$,
\[
dr\bigl(W(u(s,t))\bigr)=0.
\]
Moreover, since $\partial_s \lambda_s =0$ for $s \notin [-s_0, s_0]$, 
\[
i_{W(u(s,t))}d\hat{\lambda}=0.
\]
Hence
\[
\Delta r(s,t) = \big\lvert \nabla_s u(s,t) \big\rvert_{s,t}^2 + \partial_s a(s) \cdot r(s,t) \ge 0.
\]

\textbf{Step 4.}
Next we prove the following: 
there exist $c_2, c_3>0$ which depend only on $(J_s)_s$ and $\rho$, such that if $a$ satisfies 
$\partial_s a \ge c_2a+c_3$ on $[-s_0,s_0]$, then 
$r(s,t) \le 2$ for any $(s,t) \in \R \times \R/\Z$.

Assume that $r(s,t)>2$ for some $(s,t) \in \R \times \R/\Z$.
Then, since $j_2^*W=(0,\partial_s)$ on $\bigl(\partial X \times [2,\infty)\bigr) \times \R$, 
\[
dr\bigl(W(u(s,t))\bigr)=0.
\]
Hence 
\[
\Delta r(s,t) = |\nabla_s u|_{s,t}^2 + d\hat{\lambda}\bigl(\partial_t u, W(u)\bigr) + \partial_s a \cdot r.
\]
Since $j_2^*\hat{\lambda}=r\lambda_s$, 
\[
i_{\partial_s}(j_2^*d\hat{\lambda})
=i_{\partial_s}\bigl(dr \wedge \lambda_s + r(d\lambda_s+ds \wedge \partial_s\lambda_s)\bigr) = r\partial_s\lambda_s.
\]
Hence
\[
-d\hat{\lambda}\bigl(\partial_tu, W(u)\bigr)=i_{\partial_s}(j_2^*d\hat{\lambda})(\partial_t v) 
= r\partial_s\lambda_s(\partial_t v) 
= r\partial_s\lambda_s(\partial_t z).
\]
Therefore, there exists $c_4>0$ which depends only on $(J_s)_s$ and $\rho$ such that
\[
\big\lvert d\hat{\lambda}\bigl(\partial_t u, W(u) \bigr) \big\rvert 
=r \cdot \big\lvert \partial_s\lambda_s(\partial_t z) \big\rvert
\le c_4 r^{1/2} |\partial_t u|_{s,t}.
\]
Moreover, since $u$ satisfies the Floer equation 
$\nabla_s u - J_{s,t}\partial_t u - ar\partial_r=0$, 
\[
|\partial_t u|_{s,t} \le |\nabla_s u|_{s,t} + |ar\partial_r|_{s,t}
=|\nabla_s u|_{s,t} + ar^{1/2}.
\]
Therefore
\begin{align*}
\Delta r &\ge |\nabla_s u|_{s,t}^2 + \partial_s a \cdot r - c_4\cdot r^{1/2}\bigl(|\nabla_s u|_{s,t}+ar^{1/2}\bigr) \\
         &\ge |\nabla_s u|_{s,t}^2 + \partial_s a \cdot r - \bigl(|\nabla_s u|_{s,t}^2 + c_4^2r\bigr)/2 - c_4ar \\
         &\ge (\partial_s a - c_4 a - c_4^2/2)r.
\end{align*}
Hence setting $c_2:=c_4$, $c_3:=c_4^2/2$, the following holds:
\begin{quote}
Assume $\partial_s a \ge c_2a+c_3$ on $[-s_0, s_0]$. 
Then, $\Delta r(s,t) \ge 0$ if $s \in [-s_0,s_0]$ and $r(s,t)>2$.
\end{quote}
On the other hand, by Step 3, $\Delta r(s,t) \ge 0$ if $s \notin [-s_0,s_0]$.
Hence if $\partial_s a \ge c_2a+c_3$ on $[-s_0, s_0]$, 
$\Delta r(s,t) \ge 0$ on $\{r>2\} \subset \R \times \R/\Z$.
This implies that $\{r>2\} = \emptyset$, by same arguments as step 1.

\textbf{Step 5.}
Finally we prove that there exist $c_0, c_1>0$, which depend only on $(J_s)_s$ and $\rho$, 
such that if $a$ satisfies $\partial_s a \ge c_0 a + c_1$ on $[-s_0,s_0]$, then 
$\Delta r \ge 0$.

It is convinient to equip $E$ with a Riemannian metric $\langle\,,\,\rangle_t$ such that 
\begin{itemize}
\item $\langle\,,\,\rangle_t|_{E_s}$ is equal to $\langle\,,\,\rangle_{s,t}$.
\item For any $q \in E_s$, $T_qE_s$ and $W_q$ is orthogonal with respect to $\langle\,,\,\rangle_t$.
\item For any $q \in E$, $|W_q|_t=1$.
\end{itemize}

By step 4, $r(s,t) \le 2$ for any $(s,t) \in \R \times \R/\Z$.
Therefore there exist $c_5,c_6,c_7>0$ such that
\begin{align*}
\big\lvert d\hat{\lambda}(\partial_t u, W(u)) \big\rvert &\le c_5 |\partial_t u|_t, \\
\big\lvert dr\bigl(W(u)) \big\rvert &\le c_6, \\
\big\lvert \partial_s\bigl(dr(W(u))\bigr) \big\rvert &\le c_7 |\partial_s u|_t
\end{align*}
on $[-s_0, s_0] \times \R/\Z$. On the other hand, 
\begin{align*}
|\partial_t u|_t &\le |\nabla_s u|_t + |ar\partial_r|_t = |\nabla_s u|_t + ar^{1/2} \le |\nabla_s u|_t + a\sqrt{2}, \\
|\partial_s u|_t &\le |W(u)|_t + |\nabla_s u|_t =1+|\nabla_s u|_t.
\end{align*}
Hence, 
\begin{align*}
\Delta r &\ge |\nabla_s u|_t^2 + \partial_s a \cdot r - c_5\bigl(|\nabla_s u|_t+ a\sqrt{2}\bigr) 
                             - c_6 a - c_7\bigl(1+|\nabla_s u|\bigr) \\
&=|\nabla_s u|_t^2 + \partial_s a \cdot r - (\sqrt{2}c_5+c_6)a - (c_5+c_7)|\nabla_s u|_t - c_7.
\end{align*}
Therefore, setting $c_0:=\sqrt{2}c_5+c_6$, $c_1:=c_7+(c_5+c_7)^2/2$,
$\Delta r \ge \partial_s a - c_0 a -c_1$ (we use $\partial_s a \ge 0$ and $r \ge 1$).
Hence setting $c_0$ and $c_1$ as above, the following holds:
\begin{quote}
If $\partial_s a \ge c_0a+c_1$ on $[-s_0,s_0]$, then $\Delta r \ge 0$ on $[-s_0,s_0] \times \R/\Z$.
\end{quote}
On the other hand, $\Delta r(s,t) \ge 0$ when $s \notin [-s_0,s_0]$, as we have proved in step 3.
This completes the proof of step 5.
\end{proof}

\begin{cor}\label{cor:C0bound-3}
Let $c_0, c_1>0$ satisfy the condition in lemma \ref{lem:C0bound-3}, and 
assume that $\partial_s a \ge c_0 a +c_1$ on $[-s_0, s_0]$.
Then, for any $x_- \in \mca{P}(H_-)$, $x_+ \in \mca{P}(H_+)$ and 
$u \in \hat{\mca{M}}\bigl(x_-, x_+ :(H_s,J_s)_s\bigr)$, 
\[
\partial_s \mca{A}_{H_s}\bigl(u(s):X,\lambda_s\bigr) \le 0.
\]
In particular, if $\mca{A}_{H_+}(x_+)>\mca{A}_{H_-}(x_-)$, then 
$\hat{\mca{M}}\bigl(x_-, x_+:(H_s,J_s)_s\bigr)=\emptyset$.
\end{cor}
\begin{proof}
By lemma \ref{lem:C0bound-3}, $u(\R \times \R/\Z) \subset j_1(X \times \R)$.
Since $j_1^*W = (0,\partial_s)$, 
\[
\partial_s\mca{A}_{H_s}\bigl(u(s):X,\lambda_s\bigr)
=\int_{\R/\Z} u(s)^* \partial_s \lambda_s 
-\int_{\R/\Z} \bigl\{ |\nabla_s u(s,t)|^2 + \partial_s H_{s,t}\bigl(u(s,t)\bigr) \bigr\} dt.
\]
Since $u(s):\R/\Z \to X$ is contractible, one can extend $u(s)$ to 
$\overline{u(s)}: D^2 \to X$ so that $\overline{u(s)}(e^{2\pi i \theta}) = u(s)(\theta)$. Then, 
\[
\int_{\R/\Z} u(s)^* \partial_s\lambda_s = \int_{D^2} \overline{u(s)}^* \partial_s(d\lambda^s) =0.
\]
Hence 
\[
\partial_s\mca{A}_{H_s}\bigl(u(s):X,\lambda_s\bigr) 
=-\int_{\R/\Z} \bigl\{ |\nabla_s u(s,t)|^2 + \partial_s H_{s,t}\bigl(u(s,t)\bigr) \bigr\} dt \le 0,
\]
where the last inequality follows from $\partial_s H_{s,t} \ge 0$ on $X$.
\end{proof}

Finally we prove theorem \ref{thm:truncated}.
Define a morphism 
$\psi: C_*\bigl(H_-,\partial_{H_-,J_-} \bigr) \to C_*\bigl(H_+,\partial_{H_+,J_+} \bigr)$
by 
\[
\psi\bigl([x_-]\bigr)= \sum_{x_+ \in \mca{P}_k(H_+)} \sharp\hat{\mca{M}}\bigl(x_-,x_+:(H_s,J_s)_s\bigr)\cdot [x_+] \qquad \bigl(x_- \in \mca{P}_k(H_-)\bigr).
\]
Then, by lemma \ref{lem:C0bound-3}, $\psi$ is a chain map. 
Moreover, by corllary \ref{cor:C0bound-3}, 
$\psi$ defines a chain map 
$\psi^{<a}: C_*^{<a}\bigl(H_-,\partial_{H_-,J_-}\bigr) \to C_*^{<a}\bigl(H_+,\partial_{H_+,J_+}\bigr)$ for any 
$a \in (0,\infty]$.
This defines a morphism $\HF_*^{<a}(H_-) \to \HF_*^{<a}(H_+)$, and 
we denote this morphism also by $\psi^{<a}$. 
It is clear from the construction that 
\[
\xymatrix{
\HF_*^{<a}(H_-:X,\lambda)\ar[r]^{\psi^{<a}}\ar[d]&\HF_*^{<a}(H_+:X,\lambda')\ar[d] \\
\HF_*^{<b}(H_-:X,\lambda)\ar[r]_{\psi^{<b}}&\HF_*^{<b}(H_+:X,\lambda')
}
\]
commutes for any $a \le b$.

By taking a limit, we get a morphism $\SH_*^{<a}(X,\lambda) \to \SH_*^{<a}(X,\lambda')$ (still denoted by $\psi^{<a}$), 
and the following diagram commutes:
\[
\xymatrix{
\SH_*^{<a}(X,\lambda)\ar[r]^{\psi^{<a}}\ar[d]&\SH_*^{<a}(X,\lambda')\ar[d] \\
\SH_*^{<b}(X,\lambda)\ar[r]_{\psi^{<b}}&\SH_*^{<b}(X,\lambda')
}.
\]
It is easy to check that $\psi^{<a}$ does not depend on choices of $\rho$ and $(J^s)_s$, 
$\psi^{<a}:\SH^{<a}(X,\lambda) \to \SH^{<a}(X,\lambda)$ is the identity, and 
the following diagram commutes:
\[
\xymatrix{
\SH_*^{<a}(X,\lambda)\ar[rd]_{\psi^{<a}}\ar[rr]^{\psi^{<a}}&&\SH_*^{<a}(X,\lambda')\ar[ld]^{\psi^{<a}} \\
&\SH_*^{<a}(X,\lambda'')&
}
\]
Then, it follows that $\psi^{<a}$ is isomorphic.
Hence this completes the proof of theorem \ref{thm:truncated}.


\begin{thebibliography}{}
%
%
\bibitem{AlbersMazzucchelli}
Albers, P., Mazzucchelli, M.: Periodic bounce orbits of prescribed energy, Int. Math. Res. Notices, 
doi:10.1093/imrn/rnq193 (2010)
\bibitem{BenciGiannoni}
Benci, V., Giannoni, F.: Periodic bounce trajectories with a low number of bounce points, 
Ann. Inst. Henri Poincar\'{e}, Anal. Non Lin\'{e}aire, 6, No.1, 73--93 (1989)
\bibitem{Cieliebak}
Cieliebak, K.: Handle attaching in symplectic homology and the Chord Conjecture, 
J. Eur. Math. Soc, 4, 115--142 (2002)
\bibitem{CFHW}
Cieliebak, K., Floer, A., Hofer, H., Wysocki, K.: Applications of symplectic homology II: Stability of the action spectrum, 
Math. Z, 223, 27--45 (1996)
\bibitem{Long}
Long, Y.: Index theory for symplectic paths with applications, Progr. Math, vol.207, Birkh\"{a}user, Basel (2002)
\bibitem{Oancea}
Oancea, A.: A survery of Floer homology for manifolds with contact type boundary or symplectic homology, 
Ensaios Math, 7, 51--91 (2004)
\bibitem{Viterbo}
Viterbo, C.: Functors and computations in Floer homology I, Geom. Funct. Anal, 9, 985--1033 (1999)
\bibitem{Viterbo2}
Viterbo, C.: Metric and isoperimetic problems in symplectic geomtery, J. Amer. Math. Soc, 13, No.2, 411--431 (2000)
\bibitem{Weinstein}
Weinstein, A.: Contact surgery and symplectic handlebodies, Hokk. Math. J, 20, 241--251 (1991)
\end{thebibliography}


\end{document}